\documentclass[11pt, a4paper]{article}

\usepackage{amsmath}\usepackage{amsfonts}\usepackage{amssymb}\usepackage{amsthm}
\usepackage{graphicx}
\usepackage{ifthen}
\usepackage{lscape}
\usepackage{subfigure}
\usepackage{setspace}
\usepackage{subeqnarray}
\usepackage{url}

\newenvironment{keywords}{\footnotesize{\bf Keywords: }}{}
\newenvironment{AMS}{\footnotesize{\bf AMS subject classification: }}{}
\usepackage[left=1in,top=1.5in,right=1in,bottom=1in]{geometry}

\usepackage{verbatim}
\usepackage{mdwlist}
\usepackage{xcolor}
\usepackage{hyperref}
\usepackage{bm}
\usepackage{ulem}

\newtheorem{theorem}{Theorem}[section]
\newtheorem{lemma}[theorem]{Lemma}

\newtheorem{proposition}[theorem]{Proposition}
\newtheorem{remark}{Remark}[section]
\numberwithin{equation}{section}

\long\def\symbolfootnote[#1]#2{\begingroup\def\thefootnote{\fnsymbol{footnote}}
\footnote[#1]{#2}\endgroup}

\renewcommand{\paragraph}[1]{}
\renewcommand{\paragraph}[1]{}
\usepackage{verbatim}
\usepackage{ifpdf}
\ifpdf
\else
\renewcommand{\includegraphics}[1]{\framebox{Graphics Placeholder}}
\fi

\author{Assyr Abdulle\thanks{ANMC, Section of Mathematics, Swiss Federal Institute of Technology (EPFL), CH-1015, Lausanne,
Switzerland% (assyr.abdulle@epfl.ch).%
        }
        \and
		Ping Lin\thanks{Division of Mathematics, University of Dundee, Old Hawkhill, Dundee, DD1 4HN, UK%
        }
        \and
        Alexander V. Shapeev$^{*,}$\thanks{Present address: 
                School of Mathematics, 206 Church St.~SE, University of Minnesota, Minneapolis,
                MN 55455, US}
}

\title{A priori and a posteriori $W^{1,\infty}$ error analysis of a QC method for complex lattices}

\newcommand{\smfrac}[2]{{\textstyle\frac{#1}{#2}}}
\newcommand{\bbR}{{\mathbb R}}
\newcommand{\bbZ}{{\mathbb Z}}

\newcommand{\calE}{{\mathcal E}}
\newcommand{\calI}{{\mathcal I}}
\newcommand{\calL}{{\mathcal L}}
\newcommand{\calN}{{\mathcal N}}
\newcommand{\calU}{{\mathcal U}}
\newcommand{\calR}{{\mathcal R}}

\newcommand{\calT}{{\mathcal T}}
\newcommand{\calP}{{\mathcal P}}

\newcommand{\eps}{{\epsilon}}
\newcommand{\Const}{{\rm Const}}
\newcommand{\interior}{{\rm int}}

\def\del{\delta\hspace{-1pt}}
\def\ddel{\delta^2\hspace{-1pt}}

\def\transpose{{\!\top}}

\newcommand{\delE}{\del E}
\newcommand{\ddelE}{\ddel E}

\newcommand{\delPhi}{\del\Phi}
\newcommand{\ddelPhi}{\ddel\Phi}

\newcommand{\C}{{\rm C}}

\renewcommand{\c}{{\rm c}}

\newcommand{\aux}{{\rm aux}}

\newcommand{\<}{\langle}
\renewcommand{\>}{\rangle}
\newcommand{\dd}{{\rm d}}

\newcommand{\dx}{\dd x}

\newcommand{\wrt}{w.r.t.\ \relax}
\newcommand{\cf}{cf.\ \relax}

\begin{document}

% \sloppy
\allowdisplaybreaks
\maketitle

\begin{abstract}
In this paper we prove a priori and a posteriori error estimates for a multiscale numerical method for computing equilibria of multilattices under an external force.
The error estimates are derived in a $W^{1,\infty}$ norm in one space dimension.
One of the features of our analysis is that we establish an equivalent way of formulating the coarse-grained problem which greatly simplifies derivation of the error bounds (both, a priori and a posteriori).
We illustrate our error estimates with numerical experiments.
\end{abstract}



\pagestyle{myheadings}
\thispagestyle{plain}
%\markboth{A.~Abdulle, P.~Lin, A.~V.~Shapeev}{DISCRETE HOMOGENIZATION}

%\setcounter{tocdepth}{4}
%\tableofcontents

\section{Introduction}
% discuss complex crystals, QC, homogenization; essentially a copy-paste from the existing Introduction

Multiscale methods for modelling and simulation of microscopic features in crystalline materials have been very attractive  to researchers of material sciences and applied mathematics in past two decades.
In these modelling methods it is assumed that there is an underlying atomistic model which is the ``exact" description of a material associated with certain lattice structure.
Direct atomistic simulations using the ``exact" model may not be feasible because of its huge number of degree of freedoms.
The quasicontinuum (QC) approximation is a popular method to dramatically reduce the degrees of freedom of the underlying atomistic model. It was put forward in \cite{TadmorPhillipsOrtiz1996} for a simple lattice system and in \cite{TadmorSmithBernsteinEtAl1999} for a complex lattice system.
Besides extensive application of the QC approximation in practical material simulations, there have been growing interest in rigorously analyzing the convergence of the QC approximation or the error between the ``correct" and the ``approximate'' solutions, see, e.g., \cite{
	DobsonLuskin2009a,
	DobsonOrtnerShapeev2012,
	EMing2007,
	GunzburgerZhang2010,
	Lin2003,
	Lin2007,
	MakridakisSuli,
	MingYang2009,
	OrtnerShapeev2011,
	OrtnerTheil,
	OrtnerZhang2011,
	VanKotenLiLuskinEtAl2012,
	VankotenOrtner2012},
as well as a number of works attemping to design more accurate coarse-grained algorithms, see, e.g., \cite{LiLuskinOrtner2011, LuskinOrtnerVankoten2011, OrtnerZhang2011, Shapeev2011, Shapeev2012}.
However, most of the works, with the exception of \cite{EMing2007} and \cite{VankotenOrtner2012}, are for crystalline materials with a simple lattice structure.

In this paper we consider a problem of equilibrium of an atomistic crystalline material with a complex lattice structure. The essential step in reducing the degrees of freedom is to coarse-grain the problem. The QC is one of the most efficient methods of coarse-graining the atomistic statics. The idea behind the QC is to introduce a piecewise affine constraints for the atoms in regions with smooth deformation and use the Cauchy-Born rule to define the energy of the corresponding groups of atoms. To formulate the QC method for crystals with complex lattice (for short, complex crystals) one must account for relative shifts of simple lattices which the complex lattice is comprised of \cite{TadmorSmithBernsteinEtAl1999}.
Our approach to model complex lattices is the framework of discrete homogenization, developed in our earlier paper \cite{AbdulleLinShapeev2010}.

We note that the idea of applying homogenization to atomistic media has appeared in the literature \cite{BaumanOdenPrudhomme2009, ChenFish2006, Chung2004, ChungNamburu2003, FishChenLi2007}.
We also note that the method considered in this paper is essentially equivalent to the QC for complex crystals, being put in the framework of numerical homogenization \cite{AbdulleLinShapeev2012}.
However, the rigorous discrete homogenization procedure and related numerical method allow us to derive error estimates for the homogenized QC method, when compared to the
solution of discretely homogenized atomistic equations.
It also allows, by a reconstruction procedure, to approximate the original full atomistic solution. To the best of our knowledge,
such error estimates are new. As in many numerical homogenization
techniques for PDEs, there is no need for our numerical approximation to derive
homogenized potential before-hand, since the effective potential is computed
on the fly (see, e.g., \cite{AEE12, EEL2007}). In addition, we note that the $H^1$ error estimates in our earlier unpublished paper \cite{AbdulleLinShapeev2010} are derived in one dimension for linear nearest neighbour interactions. In this paper we consider fully nonlinear multi-neighbour interactions which are technically much more difficult. Further, we will derive $W^{1,\infty}$ error estimates, which are more suitable for nonlinear interaction and are technically harder than those in the $H^1$ norm, and are rarely obtained even in the simple lattice case (the only estimates in $W^{1,\infty}$ norm that we know of are \cite{Lin2003,OrtnerSuli2008}).
Also, we remark that we establish an equivalence of the coarse-grained homogenized model and the atomistic homogenized model (Lemma \ref{lem:magic}), which significantly simplifies the $W^{1,\infty}$ error analysis of the QC method.
Finally, we derive both a priori and a posteriori error estimates.

The regularity results of this paper are similar to those in \cite{EMing2007}.
The main difference is that our results do not require a very high regularity of the external forces that was assumed in \cite{EMing2007} (where, essentially, the highly smooth external forces were necessary for using inverse inequalities to get a $W^{1,\infty}$ convergence from an $H^1$-stability).
%The regularity result sharp in  regularity Hence we were able to translate our regularity results directly into convergence of a numerical method.

Another related homogenization approach is the $\Gamma$ convergence (see, e.g., \cite{AlicandroCicalese2004, AlicandroCicaleseGloria2011}) which is an excellent technique of finding the effective macroscopic energy from the microscopic interaction law, but does not yield the rates of convergence of the minimizers of the microscopic model and the homogenized model.

The paper is organized as follows.
In Section \ref{sec:method_formulation} we formulate the multiscale method for multilattices and state our main assumptions.
In Section \ref{sec:inf-sup-conditions-and-regularity} we prove the inf-sup condition and regularity for the atomistic and the homogenized
equations.
In Section \ref{sec:analysis_hqc} we prove convergence of the approximate solutions to the exact ones.
Finally, in Section \ref{sec:numeric} we present numerical results that support our analysis.

\section{Method Formulation and Main Results} \label{sec:method_formulation}
In this section after introducing the principal notations used throughout the paper, we recall the 
equations for the equilibria of multilattices and describe our multiscale numerical method.
We then state our main convergence results.
\subsection{Atomistic Displacement and Function Spaces}
We consider an (undeformed) lattice of $N$ atoms, $\calL=\{\eps, 2\eps, \ldots, N\eps\}$, repeated periodically to occupy the entire $\eps\bbZ$.
%, where $p$ different species of atoms are periodically placed at the lattice sites.
%We consider an (undeformed) lattice of atoms $\calL=\{\eps, 2\eps, \ldots, N\eps\}$ with $\eps=\smfrac1N$
%The $p$ species of atoms are indexed with $\calP = \{1,2,\ldots,p\}$.
The positions of an atom $x\in\calL$ in the deformed configuration is $x+u(x)$, where $u=u(x)$ is the displacement.
We will consider only $\eps N$-periodic displacements, i.e., such that $u(x+\eps N) = u(x)$, thus effectively reducing the system to a finite number of degrees of freedom.
For convenience we choose $\eps = \smfrac1N$.
The space of $\eps N$-periodic functions is denoted as
\[
\calU(\calL)= \big\{u: \eps\bbZ\to\bbR: \ u(x) = u(x+\eps N)~\forall x\in \eps\bbZ \big\}.
\]
and its subspace of functions with zero average as
\[
\calU_{\#}(\calL)= \big\{u\in \calU(\calL):~\<u\>_\calL=0\big\},
\]
where the discrete integration (averaging) operator $\<\bullet\>_\calL$ is defined for $u\in \calU(\calL)$
by
\[
 \<u\>_\calL := \frac{1}{N}\sum_{x\in\calL} u(x).
\]
We sometimes also use the notation $\<u(x)\>_{x\in\calL}$ for $\<u\>_\calL$.
Also, for $u,v\in \calU(\calL)$ we define the pointwise product, $uv$, by
\[
uv(x) = u(x) v(x)
\quad \forall x\in \eps\bbZ,
\]
and the scalar product
\[
\<u,v\>_\calL := \<u v\>_\calL = \frac{1}{N}\sum_{x\in\calL} u(x)\, v(x).
\]

We will only consider displacements $u\in\calU_\#(\calL)$ since for more general displacements
$
u(x) = F x + \hat{u}(x),
$
with $F\in\bbR$ and $\hat{u}\in\calU_\#(\calL)$,
we can adsorb $F x$ into the reference positions as $u(x) = (x + F x) + \hat{u}(x)$ and rescale the spatial coordinate as $\hat{x} = x + Fx$.

For $u=u(x)\in \calU(\calL)$ we introduce the $r$-step discrete derivative ($r\in\bbZ,r\neq 0$),
\[
D_{x,r} u(x) := \frac{u(x+r\eps)-u(x)}{r\eps}.
\]
For $r=1$ the forward discrete derivative $D_{x,1} u$ we will sometimes simply be written as $D_x u$.
In addition to differentiation operators,
we also define for $u\in \calU(\calL)$, the translation operator
$T_x u\in \calU(\calL)$,
\[
T_x u(x) := u(x+\eps).
\]
Then the $r$-step translation ($r\in\bbZ$) can be expressed as a power of $T_x$, 
$
T_x^r u(x) = u(x+r\eps).
$
Finally, introduce an averaging operator,
\[
A_{x,r} := \frac1r \sum_{k=0}^{r-1} T_x^k,
\quad(r\in\bbZ,~r>0)
\]
so that we can write $D_{x,r} = A_{x,r} D_x$ ($r>0$).

On the function space $\calU(\calL)$ we define the family of norms
\[
\|u\|_q := \big(\<|u|^q\>_\calL\big)^{1/q}
\quad(1\leq q<\infty),
\qquad\text{and}\quad
\|u\|_\infty := \max_{x\in\calL} |u(x)|,
\]
and seminorms
\[
|u|_{m,q} := \|D^m u\|_q
\quad(1\leq q\leq \infty, ~m\in\bbZ,~m\geq 0).
\]
The seminorms $|u|_{m,q}$ are extended for negative $m$ as
\begin{align*}
|u|_{m,q} :=~&
\sup \big\{\<u,v\>_\calL \,:~ v\in\calU_\#(\calL), |v|_{-m,q'}=1 \big\}
\\~& \quad (1\leq q\leq \infty,~ (q')^2+q^2=1,~ m\in\bbZ,~ m<0).
\end{align*}
Note that $|u|_{m,q}$ are proper norms in $\calU_\#(\calL)$ for all $m\in\bbZ$.
Hence we denote spaces $\calU(\calL)$ and $\calU_\#(\calL)$, equipped with the respective norms, as $\calU^{0,q}(\calL)$ and $\calU^{m,q}_\#(\calL)$.

We will also work with the lattice $\calP=\{1,2,\ldots,p\}$.
For lattice functions $\eta=\eta(y)\in\calU(\calP)$ we define the operators ($D_y$, $D_{y,r}$, $T_y$, and $A_{y,r}$) and the norms similarly to functions in $\calU(\calP)$, noting that the lattice spacing of $\calP$ is $1$ whereas the lattice spacing of $\calL$ is $\eps$.

%The definitions of the discrete derivative and translation generalize to functions in
%$\calU(\calL_1\times\calL_2)$
%by considering the partial discrete derivative and translation operators, i.e.,
%$D_x,D_{x,r},T_{x}$ applied to $u(\bullet,y)$ and
%$D_y,D_{y,r},T_{y}$ applied to $u(x,\bullet)$.

For functions of two variables, $v=v(x,y)\in \calU(\calL)\otimes\calU(\calP)$, we will denote the full derivatives, translation, and averaging, by
$T := T_x T_y$, $D_r := \smfrac1{r \eps} (T^r-I)$, $D := D_1$, $A_r := \frac1r \sum_{k=0}^{r-1} T^k$.
Notice that the variables $x$ and $y$ are not symmetric in the definition of derivatives.
If a function does not depend on $y$ then the full derivatives coincide with the derivatives in $x$ (likewise for translation and averaging).
Hence, for functions of $x$ only, we will often omit the subscript $x$ in the operators $D_x$, $T_x$, $A_{x,r}$.

The following lemma, whose proof is straightforward and will be omitted, collect the useful facts about the above operators
%TODO: formatting
\begin{lemma}\label{lem:operator-facts}
(a) %For operators on $\calU(\calL)$
For any $v\in\calU(\calL)$, $r\in\bbZ$, $r>0$
the following estimates hold:
\begin{align}
%\notag
%\|T\|_q \leq~& 1
%\\
%\label{eq:A_r_estimate}
%\|A_r v\|_q \leq~& \|v\|_q
%%\quad(r\in\bbZ,~r>0)
%,
%\\
\label{eq:D_r_estimate}
\|D_r v\|_q = \|A_r D v\|_q \leq~& \|D v\|_q
%\quad(r\in\bbZ,~r>0)
,
\\ \label{eq:Dr-D_estimate}
\|D_r v - D v\|_q = \frac{\eps}r \bigg\|\sum_{k=1}^{r-1} k D_k D v\bigg\|_q \leq~& \smfrac12 \eps (r-1)\, \|D^2 v\|_q
%\\ ~& \label{eq:Dr-D_estimate}
%\quad(v\in\calU(\calL),~r\in\bbZ,~r>0)
.
\end{align}

(b) For any $v\in\calU(\calL)\otimes\calU(\calP)$, $r\in\bbZ$, $r>0$ the following estimate holds:
\begin{equation} \label{eq:Ar_Ayr_estimate}
\|A_r v - A_{y,r} v\|_q 
=
\frac{\eps}r \bigg\|
	\sum_{k=1}^{r-1} k D_{x,k} T_y^k v
\bigg\|
\leq \smfrac12 \eps (r-1)\, \|D_x v\|_q
%\\ ~&
%\quad(v\in\calU(\calL)\otimes\calU(\calP),~r\in\bbZ,~r>0)
.
\end{equation}
%
%(c) (Discrete integration by parts.)
%For $u,{v}\in \calU(\calL)$ the following identity holds:
%\[
%\<u, D_r {v} \>_\calL
%=
%- \< T^{-r} D_ru, {v} \>_\calL.
%\]
%This identity can be written in an operator form as $(D_r)^\transpose = -T^{-r} D_r$.
\end{lemma}

\subsection{Atomistic Interaction and Equilibrium}

The energy of interaction of two atoms, $x\in\calL$ and $x+\eps r\in\calL$ depends on three variables: the distance $u(x+r \eps)-u(x)$ between atoms $x$ and $x+\eps r$, and their positions in the reference configuration that are needed to account for different species of atoms.
We denote such energy using a family of functions $\Phi^\eps_r(D_{x,r} u(x); x)$, where, for a fixed $r\in\bbZ^+$, $\Phi^\eps_r$ is defined on (a subset of) $\bbR\times\calL$.
%\commentas{I reverted Assyr's deletion; I think it is important to mention how we encode the interaction energy since this is not standard in materials science. }
%For convenience, we define the energy of interaction using a family of functions $\Phi^\eps_r(D_{x,r} u(x); x)$, where, for a fixed $r\in\bbZ^+$, $\Phi^\eps_r$ is defined on (a subset of) $\bbR\times\calL$, and
%\[
%D_{x,r} u(x) := \frac{u(x+r \eps)-u(x)}{r\eps}
%.
%\]
%For $r=1$ we will sometimes simply write $D_{x} u(x)$.
The total interaction energy of the atomistic system is thus
\begin{equation}\label{eq:E_def}
E(u)
=
\bigg\<\sum_{r=1}^R \Phi^\eps_r(D_r u)\bigg\>_\calL
\quad=~
\bigg\<\sum_{r=1}^R \Phi^\eps_r(D_{x,r} u(x); x)\bigg\>_{x\in\calL}
,
\end{equation}
where $R$ is effectively the interaction radius (measured in the reference configuration).

%We consider only admissible displacements that average to zero,
%$
%u\in\calU_\#(\calL)
%,
%$
%since instead of assuming more general displacements
%$
%u(x) = F x + \hat{u}(x),
%$
%with $F\in\bbR$ and $\hat{u}\in\calU_\#(\calL)$,
%we can adsorb $F x$ into the reference positions as $u(x) = (x + F x) + \hat{u}(x)$ and rescale the spatial coordinate $x$.

The equations of equilibrium are thus
\begin{equation}
\label{eq:equilibrium}
\text{find $u\in\calU_\#$ s.t.:}\qquad
\<\delE(u), v\>_\calL := \frac{\dd}{\dd t} E(u+tv)\big|_{t=0} = \<f, v\>_\calL
\quad \forall v\in\calU_\#(\calL)
,
\end{equation}
where $f\in\calU_\#(\calL)$ is an external force.
Here $\delE:\calU(\calL) \to \calU(\calL)$ is the Gateaux derivative of $E:\calU(\calL)\to\bbR$.

\subsection{Multilattice and Homogenization}

The atoms $\calL$ are assumed to be of $p$ different species located periodically on $\calL$, and we assume that $N\in p\bbZ$.
We index the atom species with $\calP=\{1,2,\ldots, p\}$.
Note that a lattice functions $\eta=\eta(y)\in\calU(\calP)$ can be related to a lattice function $\eta^\eps=\eta(x/\eps) \in \calU(\calL)$.

We define $\Phi_r$ on an open subset of $\bbR\times\calP$ as $\Phi_r\big(\bullet; y\big) := \Phi^\eps_r\big(\bullet; \eps y\big)$ for a fixed $r$.
Due to periodicity of the microstructure, the dependence of $\Phi_r$ on $y$ is assumed to be $p$-periodic, i.e. $\Phi_r(z; \bullet)\in\calU(\calP)$ for all $z$.
For convenience of notations (e.g., in \eqref{eq:E_def} or \eqref{eq:Phi0}), we further identify, for a fixed $r$, the family of $p$ scalar functions $\Phi_r(\bullet; y)$, $y\in\calP$, with the function $\Phi_r : \calU(\calP)\supset U\to\calU(\calP)$ by identifying $\Phi_r(w(y), y)$ with $[\Phi_r(w)](y)$. (Here $\calU(\calP)\supset U\to\calU(\calP)$ denotes a function from an open subset $U$ of $\calU(\calP)$ with values in $\calU(\calP)$.)

We apply a homogenization to the atomistic energy to average out the microstructure; more precisely, to average out the dependence on $y\in\calP$.
The homogenized interaction (see \cite{AbdulleLinShapeev2010} for the details) is defined by
\begin{equation}
\label{eq:Phi0}
\Phi^0(z)
:=
\sum_{r=1}^R \<\Phi_r(z+ D_{y,r} \chi(z; y); y)\>_{y\in\calP}
\quad=~
\sum_{r=1}^R \<\Phi_r(z+ D_{y,r} \chi(z))\>_\calP
,
\end{equation}
where for a fixed $z\in\bbR$, $\chi(z)\in\calU_\#(\calP)$ solves the micro problem
\begin{equation}
\label{eq:chi}
\sum_{r=1}^R \<\delPhi_r(z+D_{y,r} \chi(z)), D_{y,r} \eta\>_\calP =0
\quad \forall \eta\in\calU_\#(\calP),
\end{equation}
and $\delPhi_r(z;y) = \smfrac\dd{\dd z}\Phi_r(z;y)$.

The homogenized interaction energy is $\int_{0}^1 \Phi^0(\smfrac\dd\dx u^0) \dd x$, whose discretized version is
$
E^0(u^0) := \< \Phi^0(D_x u^0) \>_{\calL}.
$
This leads to the homogenized equilibrium equations of the form
\begin{equation}
\label{eq:homogenized}
\text{find $u^0\in\calU_\#$ s.t.:}\qquad
\<\delPhi^0(D_x u^0), D_x v\>_\calL = \<f, v\>_\calL
\quad \forall v\in\calU_\#(\calL)
,
\end{equation}
or, written in a strong form,
\begin{equation}\label{eq:homogenized_strong}
\text{find $u^0\in\calU_\#$ s.t.:}\qquad
-D_x[\delPhi^0(D_x u^0)]=T_x f,
\end{equation}
where $D_x := D_{x,1}$.
To derive \eqref{eq:homogenized_strong} we should use $D_x^\transpose = -T_x^{-1} D_x$.

To extract the microstructure from the homogenized solution $u^0$, define the corrector
\begin{align}\label{eq:corr}
u^\c(x) :=~& \calI_\# \big(u^0(x) + \eps \chi^\eps(D_x u^0(x); x)\big),
\quad \text{where}
\\
\label{eq:Ih_def}
\calI_\# u :=~& u - \<u\>_\calL
\end{align}
and 
$
\chi^\eps(z; x) := \chi\big(z; \smfrac x\eps\big).
$
Application of $\calI_\#$ in the definition of $u^\c(x)$ is done for convenience so that $u^\c\in\calU_\#(\calL)$.

\subsection{HQC Formulation}
Define a triangulation of the region $(0,1]$ by introducing the nodes of triangulation $\calN_h\subset\calL$ and the elements $\calT_h$.
Each element $T\in\calT_h$ is defined by two nodes $\xi,\eta\in\calN_h$ as $T=\calL\cap[\xi,\eta)$, its {\it interior} is defined as $\interior(T) = \calL\cap(\xi,\eta)$, and its size as $h_T = \eta-\xi$.
We also define the element size function, $h\in\calU(\calL)$, so that
\begin{equation}\label{eq:h_def}
h(x) = h_T
\quad \forall x\in T.
\end{equation}

%The parameter $h$ denotes the largest element size.

We consider the coarse-grained spaces $\calU_h(\calL) \subset \calU(\calL)$ and $\calU_{h,\#}(\calL) \subset \calU_\#(\calL)$ of piecewise affine functions.
The space $\calU_h(\calL)$ can be characterized by
\begin{equation} \label{eq:Uh_characterization}
u\in\calU_h(\calL)
\qquad\Longleftrightarrow\qquad
D u(\xi-\eps)=D u(\xi) \quad\forall \xi\in\calL\setminus\calN_h
\end{equation}

We denote the nodal basis function of $\calU(\calL)$ associated with $\xi\in\calL$ as $w_\xi$, $w_\xi(x) := \delta_{x-\xi}$, where $\delta$ is the Kronecker delta.
The nodal basis function of $\calU_h(\calL)$ associated with $\xi\in\calN_h$ is denoted as $w^h_\xi(x)$.
The functions $w_\xi$, $\xi\in\calL\setminus\calN_h$, together with $w^h_\xi$, $\xi\in\calN_h$, form a basis of $\calU(\calL)$. Denote the nodal interpolant $\calI_h:\calU(\calL)\to \calU_h(\calL)$.
% and the orthogonal projection $\calI_\#:\calU(\calL)\to \calU_\#(\calL)$. %and $\calI_{h,\#} := \calI_\# \circ \calI_h$.
%We will use the fact that
%\begin{equation}\label{eq:Ihash_estimate}
% |\calI_\# v|_{1,\infty} \leq  |v|_{1,\infty} +  |\<v\>_\calL|_{1,\infty} \leq 2  |v|_{1,\infty}.
%\end{equation}

The HQC approximation to the exact atomistic problem \eqref{eq:equilibrium} is
\begin{equation}
\label{eq:coarse_equation}
\text{find $u^0_h\in\calU_{h,\#}$ s.t.:}\qquad
\<\delE^0(u^0_h), v_h\>_\calL = \<F^{h}, v_h\>_h
\quad \forall v_h \in\calU_{h,\#}^{1,1}(\calL),
\end{equation}
where $\<\bullet, \bullet\>_h$ denotes the duality pairing of $(\calU_{h,\#}^{1,1}(\calL))^*$ and $\calU_{h,\#}^{1,1}$, and $F^h \in \big(\calU_{h,\#}^{1,1}(\calL)\big)^*$ is a numerical approximation to $f\in\calU_\#^{-1,\infty}(\calL)$.
For convenience we extend $F^{h}$ on $\big(\calU_h^{1,1}(\calL)\big)^*$ by requiring $\<F^{h}, 1\>_h=0$, so that $\<F^{h}, \calI_\# v_h\>_h = \<F^{h}, v_h\>_h$ for all $v_h\in\calU_h$ (refer to \eqref{eq:Ih_def} for the definition of $\calI_\#$). A numerical corrector similar to \eqref{eq:corr} can be introduced as follows
\begin{equation}
\label{eq:corr_num}
u^\c_h := \calI_\# \big(u^0_h + \eps\chi^\eps(D u^0_h)\big).
\end{equation}

\subsection{Main results}\label{Main}
Before stating the main results, we introduce some additional notations.
For a Banach space $X$ denote $B_x(x_0,\rho) = \{x\in X \,:~ \|x-x_0\|<\rho\}$---a ball centered at $x_0$ with the radius $\rho$---and call it the neighborhood of $x_0$ with radius $\rho$.
For a mapping $f:U\to Z$ from an open subset $U\subset X$, $\del_x f(x_0)$ is its variational derivative at a point $x_0$.
When it causes no confusion, we may just write $\del f(x_0)$.
If $f:X\to \bbR$ with $X$ being a Hilbert space with the scalar product $\<\bullet, \bullet\>_X$, we identify $\del f(x_0)$ with an element of $X$ and write $\del f(x_0)x = \<\del f(x_0), x\>_X$; likewise the second derivative $\ddel f(x_0)$ will be identified with a linear mapping $X\to X$: $(\ddel f(x_0)x)x' = \<\ddel f(x_0)x, x'\>_X$.
The space of continuous mappings $f:U\to Z$, $U$ being bounded, will be denoted as $\C(U; Z)$ with the norm $\|f\|_\C:=\sup_{x\in U} \|f(x)\|$.
The space of functions whose $k$-th derivative is continuous will be denoted as $\C^k(U; Z)$ with a seminorm $|f|_{\C^k} := \|\delta^k f\|_\C$.
A space of mappings whose $k$-th derivative ($k\geq 0$) is Lipschitz continuous will be denoted as $\C^{k,1}(U; Z)$ and the smallest Lipschitz constant of the $k$-th derivative will be denoted as $|\bullet|_{\C^{k,1}}$.
In our analysis we will often use the fact that if $f\in\C^{k+1}$ then $|f|_{\C^{k,1}}=|f|_{\C^{k+1}}$.
In what follows we will express the statement ``The quantity $f$ is bounded by a constant that may depend on 
$f_1, \ldots, f_k$'' as $f\leq \Const\big(f_1, \ldots, f_k\big)$.

We make the following assumptions that will allow us to apply the framework of the implicit function theorem (refer to Appendix \ref{sec:appendix} for its precise statement).

%\noindent{\bf Assumptions.}
\subsubsection*{Assumptions}
We assume that there exists a microstructure $\chi_*=\chi_*(y)\,\in\calU_\#(\calP)$ and $\rho_\Phi$ such that:
\begin{itemize}
\item[{\bf 0.}]
	The micro-deformation $y + \chi_*(y)$ is a strictly increasing function of $y\in\bbZ$.
	This simply expresses the fact that the atoms in the reference configuration are sorted by increasing position $\eps (y + \chi_*(y))$.
\item[{\bf 1.}]
	For each $r\in\calR$ and $y\in\calP$, the interaction potential $\Phi_r(\bullet,y)$ is defined in a neighborhood $U(y)\subset\bbR$ of $D_{y,r} \chi_*(y)$ of radius $\rho_\Phi$ and $\Phi_r(\bullet,y)\in \C^{2,1}(U(y); \bbR)$.
%\end{itemize}
%\begin{itemize}
	\item[{\bf 2.}]
		$\chi_*$ satisfies
		\[
			\sum_{r=1}^R \<\delPhi_r(D_{y,r} \chi_*), D_{y,r} \eta\>_\calP =0
			\quad \forall \eta\in\calU_\#(\calP).
		\]
This assumption ensures that $\eps \chi_*(\smfrac{x}{\eps})$ is a solution to \eqref{eq:equilibrium} with $f=0$.
%\end{itemize}
%Finally, we assume the following stability condition:
%\begin{itemize}
	\item[{\bf 3.}]
	Nearest neighbor interaction dominate:
	\begin{equation}
	\label{eq:nn_dominance}
	\smfrac12 \min_y \ddelPhi_1(D_{y,1} \chi_*(y); y)-\sum_{r=2}^R \max_y |\ddelPhi_r(D_{y,r} \chi_*(y); y)| > 0.
	\end{equation}	
%	\item[{\bf 3b.}] There exists a {\it coercivity constant} $c_0>0$ such that
%	% The Hessian of $\sum_{r=1}^R \ddelPhi_r(D_{y,r} \chi_*)$ is positive definite, i.e.,
%	\begin{equation}
%	\label{eq:micro_stability}
%	\inf_{ |\eta|_{1,\infty}=1}
%	\sup_{ |\zeta|_{1,1}=1}
%	\Big<\sum_{r=1}^R \ddelPhi_r(D_{y,r} \chi_*(y); y) D_{y,r}\eta, D_{y,r}\zeta\Big>_\calP
%	\geq 2 c_0
%	,
%	\end{equation}
%	the inf-sup condition is satisfied for $\ddelE$:
%	\begin{equation}
%	\label{eq:infsup_exact}
%	\inf_{ |w|_{1,\infty}=1}
%	\sup_{ |v|_{1,1}=1}
%	\sum_{r=1}^R \<\ddelPhi_r^\eps(D_r \chi^\eps_*) D_r w, D_r v \>_\calL \geq 2 c_0
%	\end{equation}
%	with $\chi^\eps_*(x) := \chi_*\big(\smfrac x\eps\big)$, and $\ddelPhi^0(0)>0$ so that the inf-sup condition is satisfied for $\ddelE^0$:
%	\begin{equation}
%	\label{eq:infsup_homogenized}
%	\inf_{ |w|_{1,\infty}=1}
%	\sup_{ |v|_{1,1}=1}
%	\<\ddelPhi^0(0) D w, D v \>_\calL = \ddelPhi^0(0) \geq 2 c_0
%	\end{equation}
%	where $\Phi^0$ is defined by \eqref{eq:Phi0} with $\chi(0; y) = \chi_*(y)$.
\end{itemize}
%In Section \ref{sec:stability} we will show that Assumption 3a implies Assumption 3b.
%
\begin{remark}[An alternative formulation of Assumption 1]
	It is useful to note the following equivalent formulation of Assumption 1 (the equivalence can be established by a straightforward calculation):
	for each $r\in\calR$ the function $\Phi_r: \calU^{0,\infty}(\calP)\supset U \to \calU^{0,\infty}(\calP)$ is defined in a neighborhood $U$ of $\chi_*\in\calU^{0,\infty}(\calP))$ with radius $\rho_\Phi$, and $\Phi_r\in \C^{2,1}(U; \calU^{0,\infty}(\calP))$.
\end{remark}

We next state our main results.
We start with the a posteriori result.

\begin{theorem}[a posteriori estimate]\label{th:aposteriori}
Assume that the Assumptions 0,1,2,3 hold.
For all $F^h\in B_{(\calU_{h,\#}^{1,1})^*}(0,\rho_f)$, the solution $u_h^0$ to \eqref{eq:coarse_equation} exists and is unique in $B_{\calU_\#^{1,\infty}}(\chi^\eps_*,\rho_u)$.
Moreover, the following a posteriori estimate holds:
\begin{align} \notag
 |u^\c_h-u|_{1,\infty}
\leq~& \label{eq:posteriori}
\Const\big(c_0^{-1} C_\Phi^{(1,1)}\big) \max_{x\in\calN_h} |D u^0_h(x)-D u^0_h(x-\eps)|
\\~&+
c_0^{-1}  \|(h-\eps) f\|_{\infty}
+ \max_{ \substack{v_h\in\calU_{h,\#}(\calL), \\ |v_h|_{1,1}=1}} |\<F^h, v_h\>_h - \<f, v_h\>_\calL|
.
\end{align}
Here $C_\Phi^{(1,1)}
:=
\max_{y\in\calP} \sum_{r\in\calR} |r \delta \Phi_r(\bullet, y)|_{\C^{0,1}}$.
\end{theorem}
Note that the a posteriori error estimate has a form similar to the standard FEM estimates: there is a term based on the jumps of the solution across boundaries of elements, a term consisting of summation of the external force in the interior of elements, and a term accounting for an approximate summation of the external force.
It is worthwhile to note that for the fully refined mesh (i.e., where $h=\eps$), the term $ \|(h-\eps) f\|_{\infty}$ vanishes.

The following a priori error estimate will also be shown.

\begin{theorem}[a priori estimate]\label{th:apriori}
In addition to the Assumptions 0,1,2,3, assume that  
exact summation of the external force, i.e., that $\<F^h,v_h\>_h = \<f, v_h\>_\calL$.
Then, for all $f\in B_{\calU_\#^{-1,\infty}}(0,\rho_f)$,
%the solution $u^\aux$ to \eqref{eq:aux_equation} and
the solution $u_h^0$ to \eqref{eq:coarse_equation} with the exact summation of the external force $\<F^h,v_h\>_h := \<f, v\>_\calL$ exists and is unique in $B_{\calU_\#^{1,\infty}}(\chi^\eps_*,\rho_u)$.
Moreover, the following a priori estimate holds:
\begin{align*}
 |u^\c_h-u|_{1,\infty}
\leq~&
\Const\big(c_0, C_\Phi^{(1,1)}\big)  \|h f\|_{\infty}
.
\end{align*}
\end{theorem}

\section{Inf-sup conditions and regularity of for the atomistic and the homogenized equations}\label{sec:inf-sup-conditions-and-regularity}

In this section we start by showing that the Assumption 3 of Section \ref{Main}
implies the inf-sup conditions needed for the subsequent analysis.
We then establish regularity results for the atomistic solution
\eqref{eq:equilibrium}, for the 
micro problem \eqref{eq:chi},
and for the homogenized solution \eqref{eq:homogenized}. These regularity results are essential 
to derive the a priori and a posteriori error estimates.

\subsection{Inf-sup Conditions}
%\delas{We establish the following inf-sup conditions that will be used in what follows.}
\begin{lemma}\label{thm:assumption_3a_3b}
Assumption 3 implies the following assertions:
there exists a {\it coercivity constant} $c_0>0$ such that the following inf-sup conditions hold
	% The Hessian of $\sum_{r=1}^R \ddelPhi_r(D_{y,r} \chi_*)$ is positive definite, i.e.,
	\begin{align}
	\label{eq:micro_stability}
	\inf_{ \substack{\eta\in\calU_\#(\calP), \\ |\eta|_{1,\infty}=1}}
	\sup_{ \substack{\zeta\in\calU_\#(\calP), \\ |\zeta|_{1,1}=1}}
	\sum_{r=1}^R \Big<\ddelPhi_r(D_{y,r} \chi_*(y); y) D_{y,r}\eta, D_{y,r}\zeta\Big>_\calP
	\geq~& 2 c_0
	\\
	\label{eq:infsup_exact}
	\inf_{ \substack{w\in\calU_\#(\calL), \\ |w|_{1,\infty}=1}}
	\sup_{ \substack{v\in\calU_\#(\calL), \\ |v|_{1,1}=1}}
	\sum_{r=1}^R \<\ddelPhi_r^\eps(D_r \chi^\eps_*) D_r w, D_r v \>_\calL
	\geq~& 2 c_0,
	\\
	\label{eq:infsup_homogenized}
	\inf_{ \substack{w\in\calU_\#(\calL), \\ |w|_{1,\infty}=1}}
	\sup_{ \substack{v\in\calU_\#(\calL), \\ |v|_{1,1}=1}}
	\<\ddelPhi^0(0) D w, D v \>_\calL = \ddelPhi^0(0)
	\geq~& 2 c_0,
	\end{align}
	where $\chi^\eps_*(x) := \chi_*\big(\smfrac x\eps\big)$ and $\Phi^0(0)$ is defined by \eqref{eq:Phi0} with $\chi(0; y) = \chi_*(y)$.
\end{lemma}
\begin{proof}

We start with the inf-sup condition \eqref{eq:infsup_exact}.
%Now the summation is over the variable $x$.
%Similarly to the above, we have $\|A_{x,r}\|_{\infty}=1$.
%As above we denote $H_{x,r} = A_{x,r}^\transpose \mbox{diag}(\ddelPhi_r(D_r \chi^\eps_*);x) A_{x,r}$ and $H_x=\sum_{r=1}^R H_{x,r}$.
%Similarly to the argument above we can obtain that $H_x$ is diagonally dominant by Assumption 3a. 
We use the following estimate
\begin{align}
|\<\ddelPhi_r^\eps(D_r \chi^\eps_*) D_r w, D_r v \>_\calL|
\leq~& \notag
\max_x |\ddelPhi_r^\eps(D_r \chi^\eps_*; x)| \, \|D_r w\|_{\infty} \, \|D_r v\|_{1}
\\ \leq~& \label{eq:assumption_3a_3b:r_geq_2}
\max_x |\ddelPhi_r^\eps(D_r \chi^\eps_*; x)| \, \|D w\|_{\infty} \, \|D v\|_{1},
\end{align}
for all $r>1$.
For $r=1$ we use Lemma \ref{lem:sup_norm} and estimate
\begin{align} \notag
&
	\inf_{ |w|_{1,\infty}=1}
	\sup_{ |v|_{1,1}=1}
	\<\ddelPhi_1^\eps(D_r \chi^\eps_*) D w, D v \>_\calL
\\ \geq~&  \notag
	\smfrac12
	\inf_{ |w|_{1,\infty}=1}
	 \|\ddelPhi_1^\eps(D_r \chi^\eps_*) D w\|_{\infty}
\\ \geq~&  \label{eq:assumption_3a_3b:r_1}
	\smfrac12 \min_x |\ddelPhi_1^\eps(D_r \chi^\eps_*; x)|.
\end{align}
Thus, notice that \eqref{eq:infsup_exact} follows from \eqref{eq:assumption_3a_3b:r_geq_2}, \eqref{eq:assumption_3a_3b:r_1}, the assumption \eqref{eq:nn_dominance}, and the definition $\Phi_r^\eps(\bullet; x) = \Phi_r(\bullet; \smfrac x\eps)$.

Proving condition \eqref{eq:micro_stability} is in all ways similar to proving \eqref{eq:infsup_exact}, with an obvious change of spaces $\calU_\#(\calL)$ to $\calU_\#(\calP)$.

Finally, notice that \eqref{eq:infsup_homogenized} follows directly from estimating
\[
\ddelPhi^0(0) = \sum_{r=1}^R \< \ddel\Phi_r(D_r \chi_*(y);y) \>_{y\in\calP} \geq 2 c_0
\]
using \eqref{eq:nn_dominance}.
\end{proof}
\begin{remark}
The condition \eqref{eq:micro_stability} is the same as requiring that the Hessian of $\sum_{r=1}^R \ddelPhi_r(D_{y,r} \chi_*)$ is positive definite, due to equivalence of the norms on finite-dimensional spaces.
\end{remark}

The following Lemma has been used in the proof above.
\begin{lemma}\label{lem:sup_norm}
For $u\in\calU_\#$, 
\begin{equation}\label{eq:sup_norm}
	\sup_{ \substack{v\in\calU_\#(\calL), \\ |v|_{1,1}=1}} \<u, Dv\>_\calL \geq \smfrac12  \|u\|_{\infty}.
\end{equation}
\end{lemma}
\begin{proof}
Let $x_1 := {\rm argmax}|u|$.
We will assume that $u(x_1)>0$ without loss of generality (since both parts of \eqref{eq:sup_norm} are invariant \wrt changing $u$ to $-u$).
Choose $x_2$ such that $u(x_2)\leq 0$ (such $x_2$ always exists for a function with zero mean) and define $v_*$ so that
\[
Dv_*(x) = \begin{cases}
\smfrac12 & x=x_1 \\ 
-\smfrac12 & x=x_2 \\ 
0 & \text{otherwise.}
\end{cases}
\]
We obviously have $ |v_*|_{1,1}=1$ and
\[
\<u, Dv_*\>_\calL = \smfrac12 u(x_1) - \smfrac12 u(x_2) \geq \smfrac12 u(x_1) = \smfrac12  \|u\|_{\infty}
.
\]
\end{proof}

In the rest of the paper we will use \eqref{eq:micro_stability}--\eqref{eq:infsup_homogenized} instead of using Assumption 3 directly. Therefore, the regularity and convergence results of this paper would hold if the $\calU^{1,\infty}$ stability result \eqref{eq:micro_stability}--\eqref{eq:infsup_homogenized} is proved using assumptions other than Assumption 3.
Note, however, that the Assumption 3 is rather standard in the case of simple lattices (i.e., no dependence on $y$)
%, and is shown to be sharp for the case when $\ddelPhi_r(0)<0$ ($r\geq 2$) which is the case for most of empirical potentials.
and in the presence of only nearest neighbor interaction it can also be shown to be sharp.

%For $H^1$-stability, requiring \eqref{eq:micro_stability} alone is believed to be a sharp criterion for stability of $\ddelE$, but a complete proof is missing at this moment (for more details refer to \cite[Lemma 3.1]{EMing2007} and a discussion in \cite{HudsonOrtner:stab}).
%In the present work we require $\calU^{1,\infty}$-stability, a sharp description of which is not available even for the simple lattice case.
%This is why we additionally assume \eqref{eq:infsup_exact} and \eqref{eq:infsup_homogenized}.
%In \cite[Lemma 3.1]{EMing2007} a related result is claimed, however the proof contains an error which so far is fixed only in the simple lattice case (Hudson, Ortner, preprint).

%\begin{remark}
%The analysis in this paper uses IFT which guarantees existence and uniqueness of the solution in a certain neighborhood of the microstructure $\chi_*$.
%The radii of neighborhoods will depend only on coercivity constants \eqref{eq:infsup_exact} and \eqref{eq:infsup_homogenized}, and on the bound for the third derivatives of the interaction potential
%\begin{equation}
%\label{eq:bound_third_derivative}
%%\sup\big\{ \ddelPhi_r(z+D_{y,r}\chi) \,:~ r=1,\ldots,R,~  |z|<r_z,~\|\chi-\chi_*|_{1,\infty}<r_\chi\big\}
%\text{Lipschitz constant of $\ddelPhi_r$}
%\end{equation}
%\end{remark}

\subsection{Regularity results}\label{sec:analysis_atm}
%\begin{remark}
%Due to $\del_y F\in\C^{0,1}(U; Z)$, by possibly choosing a smaller neighborhood radii $\rho_x$ and $\rho_y$, one can ensure that $\|\del_y F(x,y)^{-1}\|\leq 2 \|\del_y F(x_0,y_0)^{-1}\|$ for all $x\in B_x(x_0,\rho_x)$ and $y\in B_y(y_0,\rho_y)$.
%\end{remark}
In this section we prove our main regularity results for the atomistic and homogenized solutions.
Instrumental for these results is a version of the Implicit Function Theorem (IFT) that we summarize in
the Appendix (see Theorem \ref{thm:IFT}) for the convenience of the readers.
For future use, we define
\begin{align*}
C_\Phi
:=~& \phantom{\displaystyle \max\mathstrut}
\max_{y\in\calP} \sum_{r\in\calR} |\Phi_r(\bullet, y)|_{\C^{0,1}}
,
\\
C_\Phi^{(1)}
:=~& \phantom{\displaystyle \max_{\ell=1,2}\mathstrut}
\max_{y\in\calP} \sum_{r\in\calR} |\delta\Phi_r(\bullet, y)|_{\C^{0,1}}
,
\\
C_\Phi^{(2)}
:=~&
\max_{\ell=1,2}
	\max_{y\in\calP} \sum_{r\in\calR} |\delta^\ell \Phi_r(\bullet, y)|_{\C^{0,1}}
, \text{ and recall}
\\
C_\Phi^{(1,1)}
=~& \phantom{\displaystyle \max\mathstrut}
\max_{y\in\calP} \sum_{r\in\calR} |r \delta \Phi_r(\bullet, y)|_{\C^{0,1}}.
\end{align*}
%\commentpl{If we have the Lipschitz constant $\|F\|_\C$ in IFT (b) then the subscript $\C^{0,1}$ may be changed to $\C$ in the third last line above.}

\subsubsection*{Regularity of the Micro-problem}

\begin{theorem}
\label{th:microproblem}
There exist $\rho_z>0$ and $\rho_\chi>0$ such that 
\begin{itemize}
	\item[(a)]
	For all $ |z|<\rho_z$, $\chi=\chi(z)$ satisfying \eqref{eq:chi} exists in $\C^{1,1}\big((-\rho_z,\rho_z); U\big)$, is unique within the ball $U=\{\|\chi(z)-\chi_*|_{1,\infty}<\rho_\chi\}$, and
	\begin{align} \label{eq:chi_estimate}
	|\chi|_{\C^{0,1}}
	\leq~&
	c_0^{-1} C_\Phi^{(1)}
	\\ \label{eq:chi_der_estimate}
	|\chi|_{\C^{1,1}}
	\leq~&
	\Const\big(c_0^{-1} C_\Phi^{(2)}\big)
	.
	\end{align}
	\item[(b)]
	The homogenized energy density $\Phi^0=\Phi^0(z)$ is well-defined by \eqref{eq:Phi0}, $\Phi^0\in\C^{2,1}\big((-\rho_z,\rho_z)\big)$, and
	\begin{align}
	\label{eq:Phi_0_est}
	|\Phi^0|_{\C^{0,1}}
	\leq~&
	C_\Phi
	\\ 	\label{eq:Phi_1_est}
	|\del\Phi^0|_{\C^{0,1}}
	\leq~&
	C_\Phi^{(1)} \Const\big( c_0^{-1} C_\Phi^{(1)} \big)
	\\ 	\label{eq:Phi_2_est}
	|\ddel\Phi|_{\C^{0,1}}
	\leq~&
	C_\Phi^{(2)} \Const\big( c_0^{-1} C_\Phi^{(2)} \big)
	\\ \label{eq:infsup_homogenized_extended}
	c_0 \leq~&
	\inf_{|z|<\rho_z}
	\inf_{ |w|_{1,\infty}=1}
	\sup_{ |v|_{1,1}=1}
	\<\ddelPhi^0(z) D_r w, D_r v \>_\calL
	.
	\end{align}
\end{itemize}
\end{theorem}
\begin{proof}
{\it Proof of (a)}

We will apply the IFT to the mapping
\[
F:\bbR\times \calU^{1,\infty}_\#(\calP)\to\calU^{-1,\infty}_\#(\calP)
,\quad
%(z,\chi)\mapsto \sum_{r\in\calR} \<\delPhi_r(z+D_{y,r} \chi), D_{y,r} \bullet\>_\calP
F(z,\chi) = -\sum_{r\in\calR} D_{y,-r} \delPhi_r(z+D_{y,r} \chi)
.
\]
Note that \eqref{eq:micro_stability} is exactly condition (ii) of the IFT.
Thus, to apply the IFT, we only need to establish that $F\in\C^{1,1}$.

Indeed, the following shows that $|\delta_\chi F|_{\C^{0,1}} \leq C_\Phi^{(2)}$:
\begin{align}
\notag
& |\delta_\chi F(z', \chi')-\delta_\chi F(z'', \chi'')|_{-1,\infty}
\\ =~& \notag
\sup_{ |\eta|_{1,\infty}=1}
\Big | \sum_{r\in\calR} D_{y,-r} [\ddelPhi_r(z'+D_{y,r} \chi') -\ddelPhi_r(z''+D_{y,r} \chi'')] D_{y,r}\eta \Big|_{-1,\infty}
\\ \leq~& \notag
\sup_{ |\eta|_{1,\infty}=1}
\Big \| \sum_{r\in\calR} [\ddelPhi_r(z'+D_{y,r} \chi') -\ddelPhi_r(z''+D_{y,r} \chi'')] D_{y,r}\eta \Big\|_{\infty}
\\ \leq~& \notag
\Big \| \sum_{r\in\calR} \ddelPhi_r(z'+D_{y,r} \chi')-\ddelPhi_r(z''+D_{y,r} \chi'') \Big\|_{\infty}
\\ =~& \notag
\max_{y\in\calP} \Big|\sum_{r\in\calR} \ddelPhi_r(z'+D_{y,r} \chi'(y); y)-\ddelPhi_r(z''+D_{y,r} \chi''(y); y) \Big|
\\ \leq~& \notag
\max_{y\in\calP} \sum_{r\in\calR} |\ddelPhi_r(\bullet, y)|_{\C^{0,1}} \big|(z'-z'')+D_{y,r} (\chi' - \chi'')\big|
\\ \leq~& \notag % \label{eq:microproblem_Fz_estimate}
\Big(\max_{y\in\calP} \sum_{r\in\calR}  |\ddelPhi_r(\bullet, y)|_{\C^{0,1}}\Big) \big(|z'-z''|+|\chi' - \chi''|_{1,\infty}\big)
,
\end{align}
where we used \eqref{eq:D_r_estimate} (and its consequence $ |D_r u|_{-1,\infty} \leq  \|u\|_{\infty}$ $\forall u\in\calU(\calL)$).
The bound on $|\delta_z F|_{\C^{0,1}}$ is obtained in the same manner.

%\begin{align}
%~& \notag
% |\delta_z F(z', \chi')-\delta_z F(z'', \chi'')|_{-1,\infty}
%\\ =~& \notag
%\Big | \sum_{r\in\calR} D_{y,-r} (\ddelPhi_r(z'+D_{y,r} \chi')-\ddelPhi_r(z''+D_{y,r} \chi'')) \Big|_{-1,\infty}
%\\ \leq~& \notag
%\Big \| \sum_{r\in\calR} \ddelPhi_r(z'+D_{y,r} \chi')-\ddelPhi_r(z''+D_{y,r} \chi'') \Big\|_{\infty}
%\\ =~& \notag
%\max_{y\in\calP} \Big|\sum_{r\in\calR} \ddelPhi_r(z'+D_{y,r} \chi'(y); y)-\ddelPhi_r(z''+D_{y,r} \chi''(y); y) \Big|
%\\ \leq~& \notag
%\max_{y\in\calP} \sum_{r\in\calR} |\ddelPhi_r(\bullet, y)|_{\C^{0,1}} \big|(z'-z'')+D_{y,r} (\chi' - \chi'')\big|
%\\ \leq~& \label{eq:microproblem_Fz_estimate}
%\Big(\max_{y\in\calP} \sum_{r\in\calR}  |\ddelPhi_r(\bullet, y)|_{\C^{0,1}}\Big) \big(|z'-z''|+|\chi' - \chi''|_{1,\infty}\big)
%,
%\end{align}

We hence get existence, uniqueness, and \eqref{eq:chi_der_estimate}.
Finally, \eqref{eq:chi_estimate} is obtained from $|F|_{\C^{0,1}} \leq C_\Phi^{(1)}$ which can be proved by calculations similar to the above.

{\it Proof of (b)}
Compute the first derivative:
\[
\delPhi^0(z)
=
\sum_{r=1}^R \<\delPhi_r(z+ D_{y,r} \chi(z)), 1+D_{y,r}\del\chi(z)\>_\calP
=
\sum_{r=1}^R \<\delPhi_r(z+ D_{y,r} \chi(z)), 1\>_\calP,
\]
the last step being due to \eqref{eq:chi}.
%, and notice that it does not depend on $\del\chi$.
From here we get \eqref{eq:Phi_0_est} by taking maximum over $z$ and recalling that with the assumed regularity of $\Phi^0$, we have that $|\Phi^0|_{\C^{0,1}} = \|\del\Phi^0\|_{\C}$.

The second derivative is
\[
\ddelPhi^0(z) =
\sum_{r=1}^R \<\ddelPhi_r(z+ D_{y,r} \chi(z)),1+ D_{y,r} \del\chi(z)\>_\calP
.
\]
By taking $\C$- and $\C^{0,1}$-norms of this expression we get \eqref{eq:Phi_1_est} and \eqref{eq:Phi_2_est}, respectively.

The coercivity in a neighborhood of $z=0$, \eqref{eq:infsup_homogenized_extended}, is a consequence of \eqref{eq:infsup_homogenized} and continuity of $\ddelPhi^0(z)$.
\end{proof}

\subsubsection*{Regularity of the atomistic and the homogenized problems}

Define $\chi^\eps(z;x) := \chi(z;x/\eps)$ and $\chi^\eps_*(x) := \chi_*(x/\eps)$.
We fix $\rho_z$ and $\rho_\chi$ as given by the Theorem \ref{th:microproblem} and moreover assume that $\rho_\chi$ is chosen such that $\rho_\chi \leq \Const(1)$.

\begin{theorem}
\label{th:macroproblem}
There exist $\rho_f>0$ and $\rho_u>0$ such that:
\begin{itemize}
\item[(a)] For all $f\in B_{\calU_\#^{-1,\infty}}(0,\rho_f)$, the solution $u$ of \eqref{eq:equilibrium} exists and is unique in $B_{\calU_\#^{1,\infty}}(\chi^\eps_*,\rho_u)$.
	Moreover, $u=u(f) \in \C^{1,1}(B_{\calU_\#^{-1,\infty}}(0,\rho_f); B_{\calU_\#^{1,\infty}}(\chi^\eps_*,\rho_u))$,
	\begin{align*}
		\|\del_f u\|_\C
		\leq~&
		c_0^{-1}
		,\qquad\text{and}
		\\
		|\del_f u|_{\C^{0,1}}
		\leq~&
		\Const\big(c_0, C_\Phi^{(2)}\big)
		.
	\end{align*}

\item[(b)] For all $f\in B_{\calU_\#^{-1,\infty}}(0,\rho_f)$, the solution $u^0$ of \eqref{eq:homogenized} exists and is unique in $B_{\calU_\#^{1,\infty}}(0,\rho_u)$.
	Moreover, $u^0=u^0(f) \in \C^{1,1}(B_{\calU_\#^{-1,\infty}}(0,\rho_f); B_{\calU_\#^{1,\infty}}(0,\rho_u))$ and
	\begin{align} \notag
		\|\del_f u\|_\C
		\leq~&
		c_0^{-1}
		,
		\\ \notag
		|\del_f u^0|_{\C^{0,1}}
		\leq~&
		\Const\big(c_0, C_\Phi^{(2)}\big)
		,\qquad\text{and}
		\\
		\label{eq:D2u0_estimate}
		 |u^0(f)|_{2,\infty}
		\leq~&
		c_0^{-1}\,  \|f\|_{\infty}
		\quad \forall f\in B_{\calU_\#^{-1,\infty}}(0,\rho_f)
		.
	\end{align}
In addition, the corrected solution $u^\c = \calI_\# (u^0 + \eps \chi^\eps(Du^0))$ is within $B_{\calU_\#^{1,\infty}}(\chi^\eps_*,\rho_u)$.
	
\item[(c)] The following estimates hold:
	\begin{align}
	\label{eq:ucu_estimate}
	2  \|u^\c-u\|_{\infty}
		\leq~&  |u^\c-u|_{1,\infty}
		\leq \eps \,\Const\big(c_0^{-1} C_\Phi^{(1,1)}\big)\,  |u^0|_{2,\infty}.
	\\
	\label{eq:u0u_estimate}
	\|u^0-u\| \leq~& \eps \,\Const\big(c_0^{-1} C_\Phi^{(1,1)}\big)\,  |u^0|_{2,\infty} + \eps\,\Const(p).
	\end{align}
\end{itemize}
\end{theorem}
\begin{proof}
{\it Proof of (a)} consists in a direct application of the IFT to $(f,u)\mapsto \delE(u)-f$.
Assumption \eqref{eq:infsup_exact} guarantees the condition (ii) of the IFT; and by doing a straightforward calculation, similar to those in part (a) of Theorem \ref{th:microproblem}, one can show the necessary regularity of this map.
Finally, one should notice that $(0, \chi^\eps_*)\mapsto 0$.

{\it Proof of (b)}. It is a standard result (cf., e.g., \cite{OrtnerSuli2008}). The proof of all the statements except \eqref{eq:D2u0_estimate} again consists in a direct application of the IFT to $(f,u^0)\mapsto \delE^0(u^0)-f$ and in all way similar to the proof of (a).

To prove \eqref{eq:D2u0_estimate}, we use coercivity of the homogenized problem, \eqref{eq:infsup_homogenized_extended}.
For a fixed $x\in\calL$ choose $\theta\in{\rm conv}\{Du^0(x), Du^0(x+\eps)\}$ such that
$\delPhi^0(Du^0(x+\eps))-\delPhi^0(Du^0(x)) = \ddelPhi^0(\theta) (Du^0(x+\eps)-Du^0(x))$.
By construction $\rho_u\leq \rho_z$, hence $\ddelPhi^0(\theta)\geq c_0>0$, therefore
\[
c_0\, |D^2 u^0(x)|
\leq
D \delPhi^0(Du^0(x); x)
= -T f(x),
\]
where we used \eqref{eq:homogenized_strong}, 
which upon taking maximum over $x$ immediately yields \eqref{eq:D2u0_estimate}.

The possibility of choosing $\rho_u$ such that $u^\c \in B_{\calU_\#^{1,\infty}}(\chi^\eps_*,\rho_u)$ follows from $ |\chi(z)-\chi_*|_{1,\infty}<\rho_\chi$ for all $|z|<\rho_z$ which is guaranteed by Theorem \ref{th:microproblem}.

{\it Proof of \eqref{eq:ucu_estimate}.}
%We will adopt the notation $D := D_x T_y + \eps^{-1} D_y$ and $D_r := D_{x,r} T_y^r + \eps^{-1} D_{y,r}$ for functions of both $x$ and $y$.
%It is consistent with the earlier convention that $D=D_x$ and $D_r = D_{x,r}$ for functions of $x$ only.

The first estimate in \eqref{eq:ucu_estimate} is the Poincar\'e inequality (see, e.g., \cite[Appendix A]{OrtnerSuli2008}), so we only need to prove the second estimate.
We start with using coercivity of $\delE$ and the fact that $u$ and $u^0$ are solutions to \eqref{eq:equilibrium} and \eqref{eq:homogenized}:
\[
c_0  |u^\c-u|_{1,\infty}
\leq
 |\ddelE(\theta) (u^\c-u)|_{-1,\infty}
=
 |\delE(u^\c)-\delE(u)|_{-1,\infty}
=
 |\delE(u^\c)-\delE^0(u^0)|_{-1,\infty}
,
\]
with some $\theta\in{\rm conv}\{u^\c, u\}\subset B_{\calU_\#^{1,\infty}}(\chi^\eps_*,\rho_u)$.
Thus we reduced the problem to estimating the consistency error, $ |\delE(u^\c)-\delE^0(u^0)|_{-1,\infty}$.

Compute $\delE(u^\c)$:
\begin{align}
\<\delE(u^\c), v\>_\calL
=~& \notag
\sum_{r=1}^R \<\delPhi^\eps_r(D_r u^\c), D_r v\>_\calL
\\ =~& \notag
\sum_{r=1}^R \big\<\delPhi^\eps_r\big(D_r u^0 + \eps D_{r}\chi^\eps(Du^0)\big), D_r v\big\>_\calL
\\ =~& \notag
\sum_{r=1}^R \Big\<A_r^\transpose \delPhi_r\big(D_r u^0(x) + D_{y,r}\chi(Du^0(x); y)
\\ ~& \phantom{\displaystyle \sum_{r=1}^R \Big\<}
+ \eps D_{x,r}T_{y,r}\chi(Du^0(x); y) \,;\,y \big)\big|_{y=x/\eps}, D v(x)\Big\>_{x\in\calL}
\label{eq:Euc}
\end{align}
and $\delE^0(u^0)$:
\begin{align}
\<\delE^0(u^0), v\>_\calL
=~& \label{eq:E0u0}
\Big\<\Big\<\sum_{r=1}^R \delPhi_r\big(D u^0(x) + D_{y,r}\chi(D u^0(x); y)\,;\,y\big)\Big\>_{y\in\calP}, D v(x)\Big\>_{x\in\calL}.
\end{align}

Notice that $\chi(z)$ satisfies the equation $D_y^\transpose\big[\sum_{r=1}^R A_{y,r}^\transpose \delPhi_r(z + D_{y,r}\chi(z))\big]=0$, hence $\sum_{r=1}^R A_{y,r}^\transpose \delPhi_r(z + D_{y,r}\chi(z))$ is constant \wrt $y$, hence
\begin{align}
\sum_{r=1}^R A_{y,r}^\transpose \delPhi_r(z + D_{y,r}\chi(z; y); y)\big|_{y=x/\eps}
=~& \notag
\Big\<\sum_{r=1}^R A_{y,r}^\transpose \delPhi_r(z + D_{y,r}\chi(z))\Big\>_\calP
\\ =~& \label{eq:E0u0_aux}
\Big\<\sum_{r=1}^R \delPhi_r(z + D_{y,r}\chi(z))\Big\>_\calP.
\end{align}

Thus, combining \eqref{eq:Euc}, \eqref{eq:E0u0}, and \eqref{eq:E0u0_aux} yields
\begin{align*}
~& \<\delE(u^\c) - \delE^0(u^0), v\>_\calL
\\ =~&
\Big\<\sum_{r=1}^R \big[
	A_r^\transpose \delPhi_r\big(D_r u^0(x) + D_{y,r}\chi\big(Du^0(x)\big) + \eps D_{x,r}T_{y,r}\chi\big(Du^0(x)\big) \big)
	\\ ~& \phantom{\displaystyle \Big\|\sum_{r=1}^R \big[}
	- A_{y,r}^\transpose \delPhi_r\big(D u^0(x) + D_{y,r}\chi\big(D u^0(x); y\big) \,;\, y\big)
\big]_{y=x/\eps}, Dv(x)\Big\>_{x\in\calL}
\\ =:~&
\Big\<\sum_{r=1}^R \calE_r, Dv\Big\>_\calL,
\end{align*}
and hence $ |\delE(u^\c) - \delE^0(u^0)|_{-1,\infty} = \big \|\sum_{r=1}^R \calE_r\big\|_{\infty}$.

In what follows we omit the arguments of $Du^0=Du^0(x)$, $\chi(Du^0) = \chi(Du^0(x); y)$, and $\delPhi_r(\bullet) = \delPhi_r(\bullet; y)$, and likewise we omit assigning $y=\smfrac x\eps$ before taking the $\calU^{0,\infty}(\calL)$--norm.

We thus estimate:
\begin{align*}
 \|\calE_r\|_{\infty}
=~&\big\|
	A_r^\transpose \delPhi_r\big(D_r u^0 + D_{y,r}\chi\big(Du^0\big) + \eps D_{x,r}T_{y,r}\chi\big(Du^0\big) \big)
\\~& \phantom{\displaystyle \big\|}
	- A_{y,r}^\transpose \delPhi_r\big(D u^0 + D_{y,r}\chi\big(D u^0\big)\big)
\big\|_{\infty}
\\ \leq~&\big\|
	A_r^\transpose \delPhi_r\big(D_r u^0 + D_{y,r}\chi\big(Du^0\big) + \eps D_{x,r}T_{y,r}\chi\big(Du^0\big) \big)
\\~& \phantom{\displaystyle \big\|}
	-
	A_r^\transpose \delPhi_r\big(D u^0 + D_{y,r}\chi\big(D u^0\big)\big)
\big\|_{\infty}
\\ ~&+\big\|
	A_r^\transpose \delPhi_r\big(D u^0 + D_{y,r}\chi\big(D u^0\big)\big)
	-
	A_{y,r}^\transpose \delPhi_r\big(D u^0 + D_{y,r}\chi\big(D u^0\big)\big)
\big\|_{\infty}
\\ =:~&  \|\calE_r^{(1)}\|_{\infty} +  \|\calE_r^{(2)}\|_{\infty},
\end{align*}

The first term is estimated as:
\begin{align*}
 \|\calE_r^{(1)}\|_{\infty}
=~&\big\|
	A_r^\transpose \delPhi_r\big(D_r u^0 + D_{y,r}\chi\big(Du^0\big) + \eps D_{x,r}T_{y,r}\chi\big(Du^0\big) \big)
\\~& \phantom{\displaystyle \big\|}
	-
	A_r^\transpose \delPhi_r\big(D u^0 + D_{y,r}\chi\big(D u^0\big)\big)
\big\|_{\infty}
\\ \leq~&\big\|
	\delPhi_r\big(D_r u^0 + D_{y,r}\chi\big(Du^0\big) + \eps D_{x,r}T_{y,r}\chi\big(Du^0\big) \big)
\\~& \phantom{\displaystyle \big\|}
	-
	\delPhi_r\big(D u^0 + D_{y,r}\chi\big(D u^0\big)\big)
\big\|_{\infty}
\\ \leq~& |\delPhi_r|_{\C^{0,1}}\, \big\|
	D_r u^0 - D u^0 + \eps D_{x,r}T_{y,r}\chi\big(Du^0\big)
\big\|_{\infty}
\\ \leq~& |\delPhi_r|_{\C^{0,1}}\, \big(
	\smfrac12 \eps(r-1)  |u^0|_{2,\infty} + \eps  |\chi|_{\C^{0,1}} |u^0|_{2,\infty}
\big)
\\ \leq~& \eps r  |\delPhi_r|_{\C^{0,1}}\, \Const\big(c_0^{-1} C_\Phi^{(1)}\big) |u^0|_{2,\infty}
%\\ \leq~&
%|\delPhi_r|_{\C^{0,1}}\, \big(
%	\smfrac12 (r-1) + |\chi|_{\C^{0,1}}
%\big)\,
%\eps  |u^0|_{2,\infty}
%\lesssim  \|f\|_{\infty}
,
\end{align*}
where in the second last step we used \eqref{eq:Dr-D_estimate} and $\|T_{y,r}\|_\infty \leq 1$.

To estimate the term with $\calE_r^{(2)}$ we use \eqref{eq:Ar_Ayr_estimate}:
\begin{align*}
 \|\calE_r^{(2)}\|_{\infty}
=~&\big\|
	A_r^\transpose \delPhi_r\big(D u^0 + D_{y,r}\chi\big(D u^0\big)\big)
	-
	A_{y,r}^\transpose \delPhi_r\big(D u^0 + D_{y,r}\chi\big(D u^0\big)\big)
\big\|_{\infty}
\\ \leq ~&
	\smfrac12\eps(r-1) \big\|D_x \delPhi_r\big(D u^0 + D_{y,r}\chi\big(D u^0\big)\big)
\big\|_{\infty}
\\ \leq ~& \smfrac12\eps(r-1)|\delPhi_r|_{\C^{0,1}}\, \big\|
	D_x ( D u^0 + D_{y,r}\chi(D u^0))
\big\|_{\infty}
\\ \leq ~& \smfrac12\eps(r-1) |\delPhi_r|_{\C^{0,1}}\, \big(|u^0|_{2,\infty}+\big\|
	D_x D_{y,r}\chi(D u^0) 
\big\|_{\infty} \big)
\\ \leq ~& \smfrac12\eps(r-1) |\delPhi_r|_{\C^{0,1}}\, \big(|u^0|_{2,\infty}+
	2  |\chi|_{\C^{0,1}} |u^0|_{2,\infty}
\big)
%\\ \leq ~& \smfrac12\eps(r-1)|\delPhi_r|_{\C^{0,1}}\, \big(1 +
%	2 |\chi|_{\C^{0,1}}
%\big)\,
% \eps  |u^0|_{2,\infty}
\\ \leq~& \eps r  |\delPhi_r|_{\C^{0,1}}\, \Const\big(c_0^{-1} C_\Phi^{(1)}\big) |u^0|_{2,\infty}
%\lesssim  \|f\|_{\infty}
.
\end{align*}
Summing the estimates for $\calE_r^{(1)}$ and $\calE_r^{(2)}$ will yield the stated result; it only remains to notice that $C_\Phi^{(1)} \leq C_\Phi^{(1,1)}$ which implies that $C_\Phi^{(1)}$ can be absorbed into $C_\Phi^{(1,1)}$.

{\it Proof of \eqref{eq:u0u_estimate}} reduces to showing $ \|\chi^\eps(Du^0)\|_{\infty}\leq\Const(p)$, since $u^0-u = (u^\c-u) - \eps\chi^\eps(Du^0)$ and $ \|u^\c-u\|_{\infty}$ has been estimated in \eqref{eq:ucu_estimate}.
We have
\[
 \|\chi^\eps(z)\|_{\infty}
=
 \|\chi(z)\|_{\infty}
\leq
 \|\chi(z)-\chi_*\|_{\infty}
+
 \|\chi_*\|_{\infty}
,
\]
where the first term can be estimated with the help of the Poincar\'e inequality 
%\commentas{I've put the referece to the Poincare inequality earlier} 
and Theorem \ref{th:microproblem}:
$ \|\chi-\chi_*\|_{\infty}
\leq \smfrac p2
 |\chi-\chi_*|_{1,\infty} 
<\smfrac p2 \rho_\chi$.

To estimate the second term, recall that due to Assumption 0, $y+\chi_*(y)$ is strictly increasing, hence $D \chi_*(y)\geq -1$ for all $y\in\bbZ$, hence using Lemma \ref{eq:chi_star_bound} we estimate $ \|\chi_*\|_{\infty} \leq \smfrac{p-1}2$.
The estimate \eqref{eq:u0u_estimate} is thus proved.
\end{proof}

\begin{lemma}\label{eq:chi_star_bound}
Let $w\in\calU_\#(\calP)$ be such that $D w(y) \geq -1$ for all $y\in\bbZ$.
Then $ \|w\|_{\infty} \leq \smfrac{p-1}2$.
\end{lemma}
\begin{proof}
We use the following representation of $w$:
\[
w(y) = \sum_{k=1}^p \big(c-\smfrac{k}{p}\big)\, D w(y-k)
,
\]
which is valid for all $c\in\bbR$.
Choose $c=1$ and estimate
\[
w(y) \geq \sum_{k=1}^p \big(1-\smfrac{k}{p}\big)\, (-1) = -\smfrac{p-1}{2}.
\]
Likewise choose $c=\smfrac1p$ and obtain the upper bound $w(y) \leq \smfrac{p-1}2$.
\end{proof}

\section{Proof of the main results}\label{sec:analysis_hqc}
In this section we prove the a posteriori and a priori error estimates.

\subsection{A Posteriori Analysis}

In order to apply our regularity results to the coarse-grained equations, we will make use of the following conjugate operator $\calI_h^*:\calU\to\calU$ as
\begin{equation}\label{eq:Istar}
\<\calI_h^* w, v\>_\calL := \<w, \calI_h v\>_h
\quad\forall v,w\in\calU.
\end{equation}
Note that $\calI_h^* w$ is supported on the nodes of the triangulation $\calN_h$ for all $w\in\calU$, and the action of $\calI_h^*$ on $w\in\calU$ can be described as distributing values of $w$ from the interior of the intervals $T\in\calT_h$ to their endpoints.

\begin{lemma}[The formulation equivalent to coarse-graining] \label{lem:magic}
The coarse-grained problem \eqref{eq:coarse_equation} is equivalent to the following (fully atomistic) problem
\begin{equation}
\label{eq:coarse_equiv_equation}
\text{find $u\in\calU_\#$ s.t.:}\qquad
\<\delE^0(u), v\>_\calL = \<\calI_h^* F^h, v\>_\calL
\quad \forall v\in\calU_\#^{1,1}(\calL)
.
\end{equation}
\end{lemma}
\begin{proof}
Using the fact that the functions $w_\xi$ for $\xi\in\calL\setminus\calN_h$, together with $w^h_\xi$ for $\xi\in\calN_h$, form a basis of $\calU(\calL)$, rewrite \eqref{eq:coarse_equation} and \eqref{eq:coarse_equiv_equation} as, respectively,
\begin{subeqnarray}\label{eq:coarse_equation_alt}
	\text{find $u\in\calU$ s.t.:}\quad & &
	u\in\calU_h
\slabel{eq:coarse_equation_alt_one}
\\
	& &
	\<\delE^0(u), w^h_\xi\>_\calL = \<F^{h}, w^h_\xi\>_h
	\quad \forall \xi\in\calN_h
\slabel{eq:coarse_equation_alt_two}
\\
	& &
	\<u\>_\calL=0
,
\slabel{eq:coarse_equation_alt_three}
\end{subeqnarray}
and
\begin{subeqnarray}\label{eq:coarse_equiv_equation_alt}
	\text{find $u\in\calU$ s.t.:}\quad & &
	\<\delE^0(u), w_\xi\>_\calL = \<\calI_h^* F^h, w_\xi\>_\calL
	\quad \forall \xi\in\calL\setminus\calN_h
\slabel{eq:coarse_equiv_equation_alt_one}
\\
	& &
	\<\delE^0(u), w^h_\xi\>_\calL = \<\calI_h^* F^h, w^h_\xi\>_\calL
	\quad \forall \xi\in\calN_h
\slabel{eq:coarse_equiv_equation_alt_two}
\\
	& &
	\<u\>_\calL=0
.
\slabel{eq:coarse_equiv_equation_alt_three}
\end{subeqnarray}
The equations \eqref{eq:coarse_equation_alt_three} and \eqref{eq:coarse_equiv_equation_alt_three} are identical. The equations \eqref{eq:coarse_equation_alt_two} and \eqref{eq:coarse_equiv_equation_alt_two} are also equivalent since $\<\calI_h^* F^h, w^h_\xi\>_\calL = \<F^{h}, \calI_h w^h_\xi\>_h = \<F^{h}, w^h_\xi\>_h$.
It thus remains to prove equivalence of \eqref{eq:coarse_equation_alt_one} and \eqref{eq:coarse_equiv_equation_alt_one}.

Fix $\xi\in\calL\setminus\calN_h$.
The right-hand side of \eqref{eq:coarse_equiv_equation_alt_one} is zero, since $\calI_h w_\xi = 0$ and hence
\[
\<\calI_h^* F^h, w_\xi\>_\calL = \<F^h, \calI_h w_\xi\>_h = 0.
\]

Evaluate the left-hand side of \eqref{eq:coarse_equiv_equation_alt_one}:
\[
0 = 
\<\delE^0(u), w_\xi\>_\calL
=
\<\delPhi^0(D u), D w_\xi\>_\calL
,
\]
which in coordinate notation reads
\begin{equation}\label{eq:magic_lemma_Dtildeu}
\delPhi^0\big(D u(\xi-\eps)\big) = \delPhi^0\big(D u(\xi)\big)
.
\end{equation}
Since $\Phi^0$ is convex (cf.\ \eqref{eq:infsup_homogenized}), \eqref{eq:magic_lemma_Dtildeu} is equivalent to $D u(\xi-\eps)=D u(\xi)$.
Since $\xi\in\calL\setminus\calN_h$ was arbitrary, it is further equivalent to $u\in\calU_h$ (cf.\ \eqref{eq:Uh_characterization}).
\end{proof}

%One choice for $\calI_h^* F^h$ is exact integration of the external force $F_h^{(1)}(v) = \<f_h^{(1)}, v\>_\calL = \<f, \calI_h v\>_\calL$. Then
%\begin{equation}\label{eq:estimate_f1}
%\big |f_h^{(1)}\big|_{-1,\infty}
%=
%\sup_{ |v|_{1,1}=1} \<f, \calI_h v\>_\calL
%\leq  |f|_{-1,\infty} \,\big |\calI_h\big|_{1,1}
%=  |f|_{-1,\infty}.
%\end{equation}
%%It is easy to see that $\big |\calI_h\big|_{1,1} = 1$: it cannot be greater than $1$ since the discrete total variation of an interpolant cannot exceed that of an original function, and at the same time it is attained on any nonzero function which is monotone on each segment of the mesh.
%Hence whenever $f\in B_{\calU_\#^{-1,\infty}}(0,\rho_f)$, $f_h^{(1)}\in B_{\calU_\#^{-1,\infty}}(0,\rho_f)$.
%
%The other choice is using the rectangle rule to approximate the quadrature: $F_h^{(2)}(v_h) = \sum_k f(v_h(x_k))$.
%As a matter of exercise, one can show that
%\[
%\big |f_h^{(2)}\big|_{-1,\infty}
%\leq 
%\big \|f_h^{(2)}\big\|_{\infty}
%\leq 
%\big \|f\big\|_{\infty}
%\]
%and also, as a matter of a harder exercise, that this bound is quasi-sharp.

Lemma \ref{lem:magic} motivates us to introduce the following auxiliary problem
\begin{equation} \label{eq:aux_equation}
\text{find $u^\aux\in\calU_\#$ s.t.:}\qquad
\<\delE(u^\aux), v\>_\calL = \<\calI_h^* F^h, v\>_\calL
\quad \forall v\in\calU_\#^{1,1}(\calL).
\end{equation}
We can then apply Theorem \ref{th:macroproblem} to \eqref{eq:coarse_equation} and \eqref{eq:aux_equation} and immediately obtain the following intermediate result:
\begin{proposition}\label{prop:HQC_intermediate}
For all $\calI_h^* F^h\in B_{\calU_\#^{-1,\infty}}(0,\rho_f)$, the solution $u^\aux$ to \eqref{eq:aux_equation} and the solution $u_h^0$ to \eqref{eq:coarse_equation} both exist and are unique in $B_{\calU_\#^{1,\infty}}(\chi^\eps_*,\rho_u)$ and $B_{\calU_\#^{1,\infty}}(0,\rho_u)$, respectively.
Moreover, the respective Lipschitz bounds $u^\aux=u^\aux(\calI_h^* F^h)\in\C^{1,1}$ and $u_h^0=u_h^0(\calI_h^* F^h)\in\C^{1,1}$, and the estimates
\begin{align}
\label{eq:u_aux_estimate}
 |u^\aux-u|_{1,\infty}
\leq~&
c_0^{-1}  |\calI_h^* F^h-f|_{-1,\infty},
\\ \label{eq:u_0_h_estimate}
 |u^0_h|_{2,\infty}
\leq~&
c_0^{-1}  \|\calI_h^* F^h\|_{\infty},
\\ \label{eq:uhc_u_aux_estimate}
 |u^\c_h-u^\aux|_{1,\infty}
\leq~&
\eps\,\Const\big(c_0^{-1} C_\Phi^{(1,1)}\big)  |u^0_h|_{2,\infty},
\end{align}
hold where $u^\c_h$ is defined in \eqref{eq:corr_num}.
\end{proposition}

It remains to further estimate the respective quantities in Proposition \ref{prop:HQC_intermediate}.

First, we notice that $\eps  |u^0_h|_{2,\infty}$ is nothing but the standard error indicator with jumps over elements.
Indeed, for an arbitrary $u_h\in\calU_{h,\#}$, we have
\begin{equation} \label{eq:aposteriori-u2}
 |u_h|_{2,\infty}
=
\max_{x\in\calL} |D^2 u_h(x)|
=
\max_{x\in\calN_h} |D^2 u_h(x-\eps)|
=
\smfrac1\eps \max_{x\in\calN_h} |D u_h(x)-D u_h(x-\eps)|
.
\end{equation}

Second, we split
\begin{align} \notag
 |\calI_h^* F^h-f|_{-1,\infty}
 \leq~&
 |\calI_h^* F^h-\calI_h^* f^h|_{-1,\infty} +  |\calI_h^* f^h-f|_{-1,\infty}
 \\ \label{eq:aposteriori-fall}
 =~&
	 \max_{ |v_h|_{1,1}=1} |\<F^h, v_h\>_h - \<f, v_h\>_\calL|
	  +  |\calI_h^* f^h-f|_{-1,\infty}
.
\end{align}
Here the first term indicates how well $F^h$ approximates the action of exact force $f$ on the finite element space $\calU_{h,\#}$. We estimate the second term using Lemma \ref{lem:f_v_Ihv_estimate}:
\begin{equation} \label{eq:aposteriori-f2}
\<\calI_h^* f,v\>_\calL-\<f,v\>_\calL
=
\<f,\calI_h v\>_\calL-\<f,v\>_\calL
=
\<f,\calI_h v - v\>_\calL
\leq
 \|(h-\eps) f\|_{\infty} |v|_{1,\infty}
.
\end{equation}

\begin{lemma}\label{lem:f_v_Ihv_estimate}
\[
\<f, v-\calI_h v\>_\calL \leq  \|(h-\eps) f\|_{\infty} |v|_{1,\infty}
\qquad \forall f\in\calU_\#,~\forall v\in\calU.
\]
where $h=h(x)$ is defined by \eqref{eq:h_def}.
\end{lemma}
\begin{proof}
We have
\begin{align*}
\<f, v-\calI_h v\>_\calL
=~&
\eps \sum_{T\in\calT_h} \sum_{x\in T} f(x) [v-\calI_h v](x)
\\ \leq~&
\eps \sum_{T\in\calT_h} \max_{x\in T} |f(x)| \sum_{x\in T} \big|[v-\calI_h v](x)\big|.
\end{align*}

Fix $T\in\calT_h$, let $\xi$ and $\eta$ ($\xi<\eta$) be the two endpoints of $T$, and estimate, for $\xi<x<\eta$,
\begin{align*}
\big| [v-\calI_h v](x) \big|
=~&
\big| v(x) - \smfrac{\eta-x}{\eta-\xi} v(\xi) - \smfrac{x-\xi}{\eta-\xi} v(\eta) \big|
\\ =~&
\big| \smfrac{\eta-x}{\eta-\xi} (v(x) - v(\xi)) - \smfrac{x-\xi}{\eta-\xi} (v(\eta)-v(x)) \big|
\\ \leq~&
|v(x) - v(\xi)| + |v(\eta)-v(x)|
\\ \leq~&
\sum_{x'\in\calL\cap[\xi,x)} |\eps Dv(x')|
+
\sum_{x'\in\calL\cap[x,\eta)} |\eps Dv(x')|
= \eps \sum_{x'\in T} |Dv(x')|
.
\end{align*}
If $x=\xi$ then obviously $[v-\calI_h v](x)=0$.

Thus,
\begin{align*}
\<f, v-\calI_h v\>_\calL
\leq~&
\eps \sum_{T\in\calT_h} \max_{x\in T} |f(x)|
	\sum_{x\in \interior(T)} \eps \sum_{x'\in T} |Dv(x')|
\\=~&
\eps \sum_{T\in\calT_h} \max_{x\in T} |f(x)|
	(h_T-\eps) \sum_{x'\in T} |Dv(x')|
\\ \leq ~&
% \leq
 \|(h-\eps) f\|_{\infty} \, \eps \sum_{T\in\calT_h} \sum_{x'\in T} |Dv(x')|
.
\end{align*}
\end{proof}

\begin{proof}[Proof of Theorem \ref{th:aposteriori}]
Using \eqref{eq:uhc_u_aux_estimate} and \eqref{eq:u_aux_estimate} we can estimate
\begin{align*}
 |u^\c_h-u|_{1,\infty}
\leq~&
 |u^\c_h-u^\aux|_{1,\infty}
+
 |u^\aux-u|_{1,\infty}
\\ \leq~&
c_0^{-1}  |\calI_h^* F^h-f|_{-1,\infty}
+
\eps\,\Const\big(c_0^{-1} C_\Phi^{(1,1)}\big)  |u^0_h|_{2,\infty}
.
\end{align*}
The proof is then completed using relations \eqref{eq:aposteriori-u2}, \eqref{eq:aposteriori-fall}, \eqref{eq:aposteriori-f2}.
\end{proof}

\subsection{A Priori Estimate}
Recall that for the a priori error estimate we assume 
the exact summation of the external force, i.e., that $\<F^h,v_h\>_h = \<f, v_h\>_\calL$.
The a priori error estimate can essentially be obtained from the a posteriori estimate \eqref{eq:posteriori} using \eqref{eq:aposteriori-u2} and \eqref{eq:u_0_h_estimate}.
We only need to estimate $ |\calI_h^* f^h|_{-1,\infty}$ and $ \|\calI_h^* f^h\|_{\infty}$ (the former is needed to quantify the condition $\calI_h^* f^h\in B_{\calU_\#^{-1,\infty}}(0,\rho_f)$) in terms of $f$.
This is done in the following lemma.

\begin{lemma}
\label{lem:apriori}
\begin{align*}
 |\calI_h^* f^h|_{-1,\infty}
\leq~&
 |f|_{-1,\infty}.
\\
 \|\calI_h^* f^h\|_{\infty}
\leq~&
\smfrac1\eps  \|h f\|_{\infty}.
\end{align*}
\end{lemma}
\begin{proof}
To prove the first estimate, we need to prove the $\calU^{1,1}$ stability of $\calI_h$:
\begin{equation}\label{eq:interpolant-W11-stability}
 |\calI_h v|_{1,1} \leq  |v|_{1,1}.
\end{equation}
To prove it, start with expressing
\[
 |\calI_h v|_{1,1}
=
\eps \sum_{T\in\calT_h} \sum_{x\in T} |D \calI_h v(x)|
.
\]
Then fix $T\in\calT_h$, let $\xi$ and $\eta$ ($\xi<\eta$) be the two endpoints of $T$, and estimate
\[
\sum_{x\in T} |D \calI_h v(x)|
=
\sum_{x\in T}  \frac{|v(\eta)-v(\xi)|}{\eta-\xi}
=
|v(\eta)-v(\xi)|
=
\bigg|
	\sum_{x\in T} Dv(x)
\bigg|
\leq
	\sum_{x\in T} |Dv(x)|
.
\]
Hence \eqref{eq:interpolant-W11-stability} follows.

Now we can easily estimate $ |\calI_h^* f^h|_{-1,\infty}$:
\[
\<\calI_h^* f^h, v\>_\calL
=
\<f, \calI_h v\>_\calL
\leq  |f|_{-1,\infty}  |\calI_h v|_{1,1}
\leq  |f|_{-1,\infty}  |v|_{1,1},
\]
hence $ |\calI_h^* f^h|_{-1,\infty} \leq  |f|_{-1,\infty}$.

To derive the second estimate, we test $\calI_h^* f^h$ with an arbitrary $v\in\calU$:
\[
\<\calI_h^* f^h, v\>_\calL
=
\<f, \calI_h v\>_\calL
=
\eps \sum_{T\in\calT_h} \sum_{x\in T} f(x) [\calI_h v](x)
\leq
\eps \sum_{T\in\calT_h} \max_{x\in T} |f(x)| \sum_{x\in T} |\calI_h v|(x)
\]

Fix $T\in\calT_h$, let $\xi$ and $\eta$ ($\xi<\eta$) be the two endpoints of $T$, and estimate
\[
\eps \sum_{x\in T} |\calI_h v|(x)
\leq
\eps \sum_{x\in T} \big(\smfrac{\eta-x}{\eta-\xi} |v(\xi)| + \smfrac{x-\xi}{\eta-\xi} |v(\eta)|\big)
\leq
h_T \big(\smfrac12|v(\xi)| + \smfrac12 |v(\eta)|\big).
\]

Thus,
\begin{align*}
\<\calI_h^* f^h, v\>_\calL
\leq~&
\sum_{T\in\calT_h} \max_{x\in T} |f(x)| h_T \big(\smfrac12|f(\xi)| + \smfrac12 |v(\eta)|\big)
\\ \leq~&
 \|h f\|_{\infty} \sum_{T\in\calT_h} \big(\smfrac12|f(\xi)| + \smfrac12 |v(\eta)|\big)
\\ =~&
 \|h f\|_{\infty} \sum_{x\in\calN_h} |v(x)|
\leq
%\\ \leq~&
 \|h f\|_{\infty} \sum_{x\in\calL} |v(x)|
=
\smfrac1\eps  \|h f\|_{\infty}  \|v\|_{1}.
\end{align*}
\end{proof}

The first estimate of the above lemma means that $f\in B_{\calU_\#^{-1,\infty}}(0,\rho_f)$ implies $\calI_h^* f^h\in B_{\calU_\#^{-1,\infty}}(0,\rho_f)$.
\begin{proof}[Proof of Theorem \ref{th:apriori}]
Follows from  \eqref{eq:posteriori} using \eqref{eq:aposteriori-u2}, \eqref{eq:u_0_h_estimate}, and Lemma \ref{lem:apriori}.
\end{proof}

\section{Numerical Examples}\label{sec:numeric}

We solve numerically several model problems to illustrate the performance of HQC.
We consider a nonlinear one-dimensional model problem (Section \ref{sec:numeric:1d}), followed by a two-dimensional linear problem (Section \ref{sec:numeric:2d}).

The aim of the numerical experiments is twofold.
First, we verify numerically the sharpness of the obtained error for the 1D case.
Second, we confirm that the HQC convergence result obtained for 1D is valid in higher dimensions.

\subsection{1D}\label{sec:numeric:1d}

In the first numerical example we solve the problem \eqref{eq:equilibrium} with the period of spatial oscillation $p=2$ and number of interacting neighbors $R=3$.
The interaction potential is chosen as the Lennard-Jones potential
\[
\Phi^\eps_r(z; x)
 = -2 \big(\smfrac{z}{l_{x/\eps}}\big)^{-6} + \big(\smfrac{z}{l_{x/\eps}}\big)^{-12}
\quad (1\le r\le R)
\]
with the varying equilibrium distance
\[
l_{y} = \left\{
	\begin{array}{lcl}
	1 & & \textnormal{$y$ is even} \\
	9/8 & & \textnormal{$y$ is odd.}
	\end{array}
\right.
\]
The number of atoms is $N=2^{14}=16384$, and the external force is taken as
\[
f(x) = 50 \sin\left(1+2\pi x\right).
\]

\begin{figure}
\begin{center}
\hfill
	\includegraphics[scale=1]{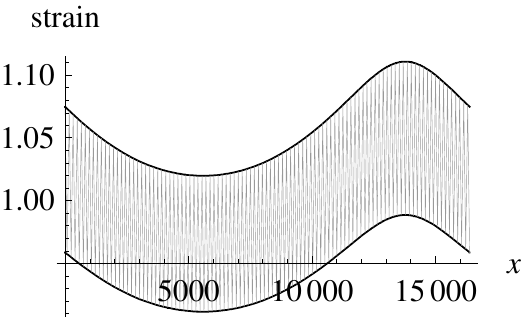}
\hfill\hfill
	\includegraphics[scale=1]{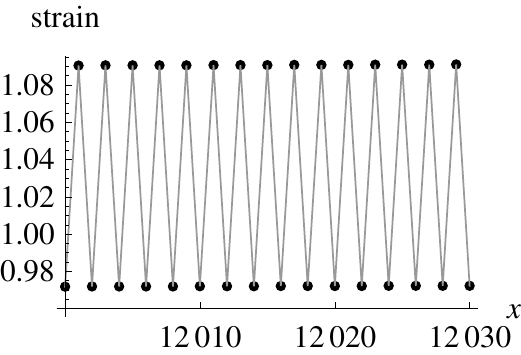}
\hfill $\mathstrut$
\end{center}
\caption{Strain $D u(x)$ of the solution of the 1D linear problem: the schematically shown complete solution (left) and the closeup of the micro-structure for 31 atoms (right).}
\label{fig:solution-linear}
\end{figure}
The (microscopic) strain $D u(x)$ for such problem is shown in Fig.\ \ref{fig:solution-linear}.

\begin{figure}
\begin{center}
\includegraphics[scale=1]{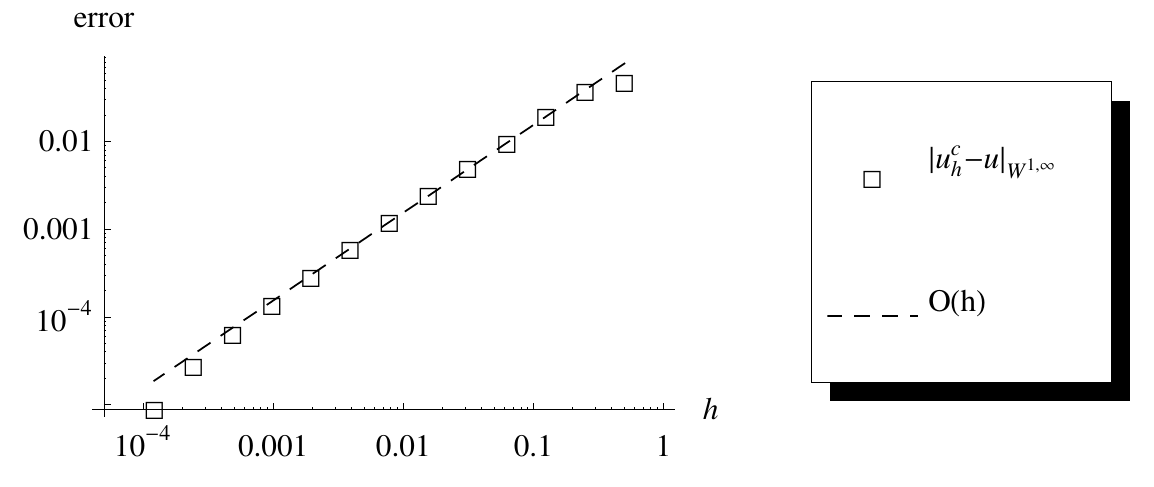}
\end{center}
\caption{Results for the 1D problem: error of the post-processed HQC solution $u_h^\c$.
The error behaves in accordance with Theorem \ref{th:apriori}.}
\label{fig:error1}
\end{figure}

Figure \ref{fig:error1} is aimed to illustrate that the estimate in Theorem \ref{th:apriori} is sharp.
Indeed, it can be seen that the corrected homogenized HQC solution $u_h^\c$ converges to the exact solution with the first order in $h$.

\subsection{2D}\label{sec:numeric:2d}

\begin{figure}
\begin{center}
\includegraphics[scale=1]{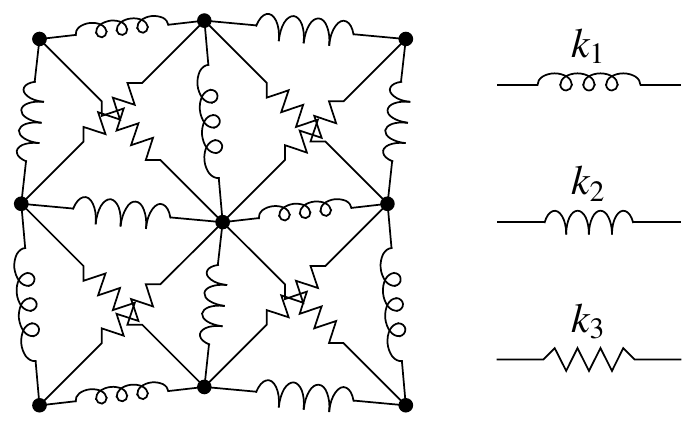}
\caption{Illustration of a 2D model problem with heterogeneous interaction.}
\label{fig:2d-springs}
\end{center}
\end{figure}

To illustrate the 2D discrete homogenization, we apply it to the following model problem.
The atomistic lattice is $\calL=(0,1]^2\cap\eps\bbZ^2$ with $\eps=1/N$, the atomistic energy is
\[
E(u) = \eps^2 \sum_{x\in\calL} \sum_{r\in\calR} \psi_{r,\frac{x}{\eps}} \smfrac12 \big(\smfrac{u(x+\eps r)-u(x)}{\eps}\big)^2,
\]
where the set of neighbors is defined by ${\mathcal R} = \left\{(1,0), (1,1), (0,1), (-1,1)\right\}$ (we omit the neighbors that can be obtained by reflection around $(0,0)$) and the interaction coefficients as
\[
\psi_{(1,1),y} = \psi_{(1,-1),y} = k_3
,\quad
\psi_{(1,0),y} = \psi_{(0,1),y} = \left\{
	\begin{array}{lcl}
	k_1 & & y_1 + y_2 \textnormal{ is even} \\
	k_2 & & y_1 + y_2 \textnormal{ is odd.}
	\end{array}
\right.
\]
Such material is illustrated in Fig.\ \ref{fig:2d-springs}.

This example was motivated by the study of Friesecke and Theil \cite{FrieseckeTheil2002}, where a similar model was considered.
Friesecke and Theil considered the model with springs similar to the one illustrated in Fig.\ \ref{fig:2d-springs}, which however was nonlinear due to nonzero equilibrium distances of the springs (so that the energy of the spring between masses $x_1$ and $x_2$ is proportional to $|x_1-x_2|^2-l_0^2$, where $l_0$ is the equilibrium distance).
They found that with certain values of parameters the lattice looses stability to non-Cauchy-Born disturbances and the lattice period doubles (thus the lattice ceases to be a Bravais lattice).

The results, given with no details of actual derivation, are the following:
The period of spatial oscillations in this case is $(2,2)$.
The function $\chi$ has the form $\chi = \chi(Y_j) = (-1)^{j_1+j_2} \frac{k_1-k_2}{4 (k_1+k_2)} I$
(here $I$ is the $2\times 2$ identity matrix).

We set the values of parameters $\epsilon=2^{-11}$, $N_1=N_2=2^{11}$, $k_1=1$, $k_2=2$, $k_3=0.25$, and the external force
\[
f(x) = 10 e^{-\cos(\pi x_1)^2-\cos(\pi x_2)^2} \left(\begin{array}{c} \sin(2\pi x_1) \\ \sin(2\pi x_2) \end{array}\right)\
-\bar{f},
\]
where $\bar{f}$ is determined so that the average of $f(x)$ is zero.
The total number of degrees of freedom of such system is approximately $8\cdot 10^{6}$.
The solution for such test case is shown in fig.\ \ref{fig:2d-solution} (the illustration is for $N_1=N_2=64$).

\begin{figure}
\begin{center}
\hfill
\includegraphics[scale=0.8]{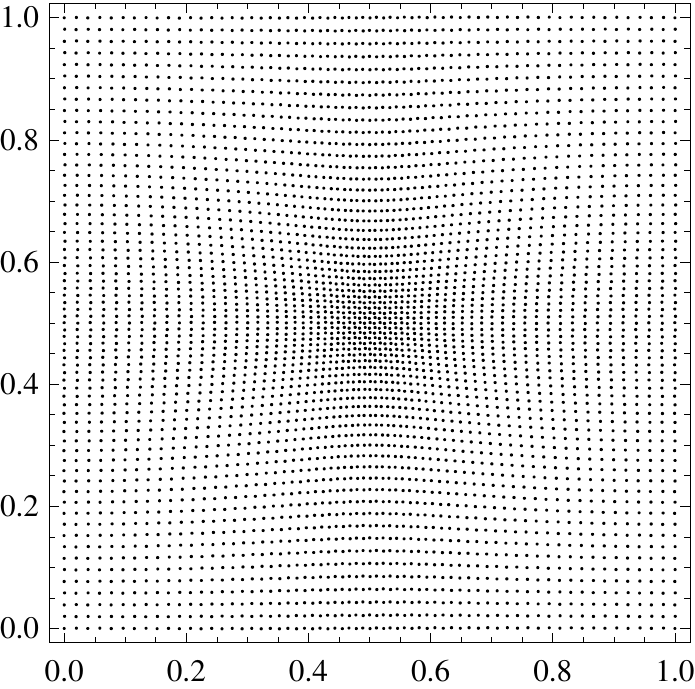}
\hfill
\includegraphics[scale=0.8]{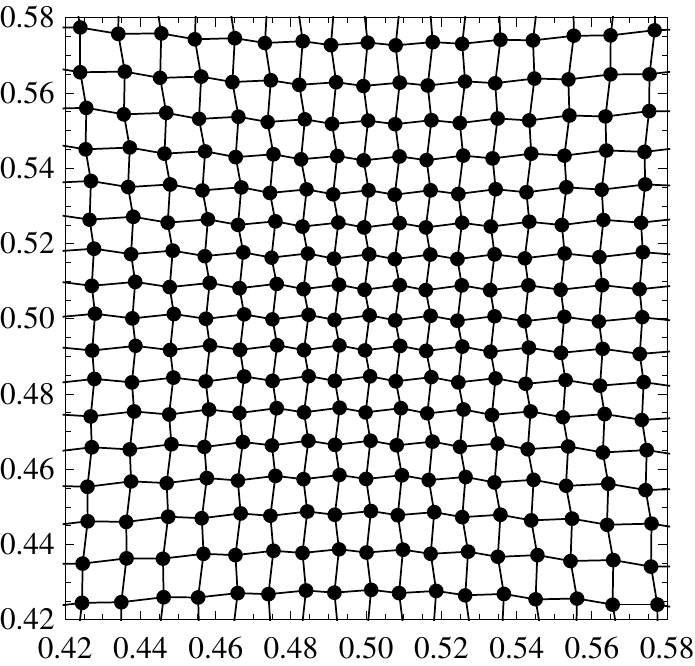}
\hfill
\end{center}
\caption{Atomic equilibrium configuration for $N_1=N_2=64$ for the 2D test case.
	Deformation of the whole material (left) and a close-up (right).}
\label{fig:2d-solution}
\end{figure}

The atomistic domain is triangulated using $t^2$ nodes and $K = 2 t^2$ triangles ($t=2,4,\ldots,2^{10}$).
In each triangle $S_k$ a sampling domain ${\mathcal I}_k$ is chosen, each sampling domain contains four atoms (see illustration in fig.\ \ref{fig:2d-triangulation}).
The number of degrees of freedom of the discretized problem is $2 t^2$.

\begin{figure}
\begin{center}
\includegraphics[scale=0.8]{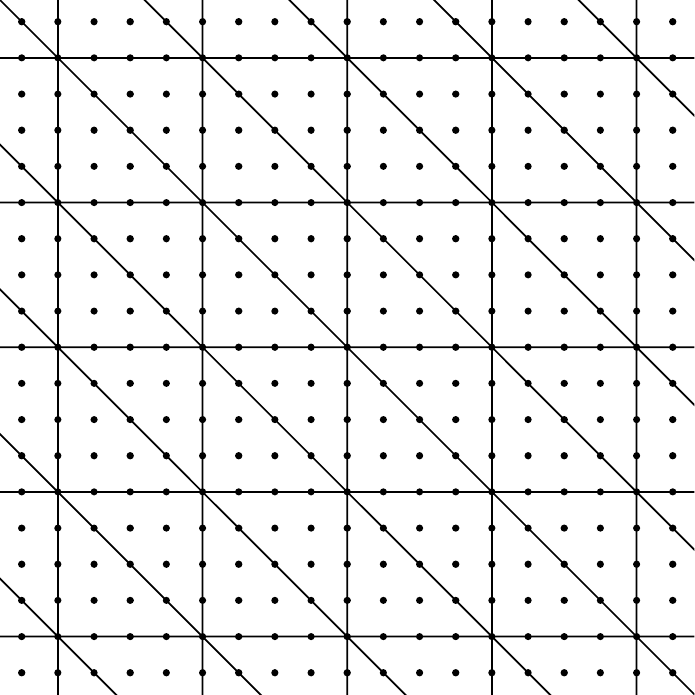}
\end{center}
\caption{Illustration of a 2D triangulation.
%	Larger atoms comprise sampling domains for HQC.
}
\label{fig:2d-triangulation}
\end{figure}

\begin{figure}
\begin{center}
\includegraphics[scale=1]{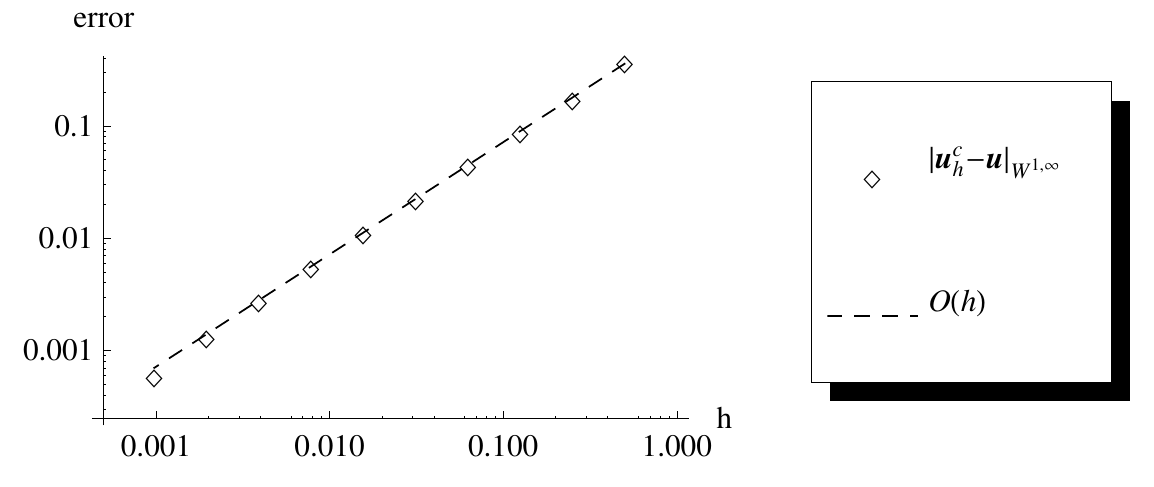}
\end{center}
\caption{Results for the 2D test case: error depending on the mesh size $h$.
The error behaves in accordance with the 1D analysis (Theorem \ref{th:apriori}).}
\label{fig:2d-testcase1-error}
\end{figure}

The error of the solution for different mesh size $h$ ($h=0.5,0.25,\ldots,2^{-10}$) is shown in Fig.\ \ref{fig:2d-testcase1-error}.
The results are essentially the same as in 1D case: the method convergences with the first order of mesh size in the $\calU^{1,\infty}$-norm.
\bigskip\\
\noindent{\bf Acknowledgements.} The work of A. Abdulle and A. V. Shapeev  was supported in part by the Swiss National Science Foundation under Grant 200021 134716/1.
The work of P. Lin was partially supported by the Leverhulme Trust Research Fellowship (No RF/9/RFG/2009/0507).
\appendix
\section{Implicit Function Theorem}\label{sec:appendix}
The following modification of the implicit function theorem (IFT) of Hildebrandt and Graves (1927), (\cf Zeidler 1986, p.~150) is used repeatedly in our analysis.
\begin{theorem}
\label{thm:IFT}
Let $X$, $Y$, and $Z$ be Banach spaces. Suppose that:
\begin{itemize}
\item[(i)] $F\in \C^{1,1}(U; Z)$ for a neighborhood $U\subset X\times Y$ of $(x_0, y_0)$ and $F(x_0, y_0)=0$.
\item[(ii)] $\del_y F(x_0,y_0)^{-1}$ exists and is bounded.
\end{itemize}
Then there exist $\rho_x>0$ and $\rho_y>0$, %that depend only on $U$, $b_0 := 2\|\del_y F(x_0,y_0)^{-1}\|$, and Lipschitz constants of $\del_x F(x,y)$, $\del_x F(x,y)$, and $\del_y F(x, y)$,
such that 
\begin{itemize}
\item[(a)] For each $x\in B_x(x_0,\rho_x)$ there exists a unique solution $y=y(x)\in B_y(y_0,\rho_y)$ of $F(x,y)=0$.
\item[(b)] %(\cf the next remark):
$y=y(x)$ is Lipschitz with the constant
\[
|y|_{\C^{0,1}} \leq b_0 \|\del_x F\|_\C \leq b_0 |F|_{\C^{0,1}},
\]
where $b_0 := 2 \|\del_y F(x_0,y_0)^{-1}\|$.
Note that $\|\del_x F\|_\C \leq |F|_{\C^{0,1}}$ due to the fact that $F$ is continuously differentiable.
\item[(c)] The derivative $\del_x y$ exists and is Lipschitz with the constant
\[
|\del_x y|_{\C^{0,1}}
\leq
\Const\big(
	b_0 \|\del_x F\|_{\C},
	b_0 |\del_x F|_{\C^{0,1}},
	b_0 |\del_y F|_{\C^{0,1}}
\big).
\]
\end{itemize}
\end{theorem}

\begin{proof}[Proof of estimates in (b) and (c)]
%$\mathstrut$ \commentas{Assyr \& Ping. Do you think the proof should be deleted or completed and shifted to an appendix? The only reason why a proof could be kept is for precise norms and constants in the estimates.
%\moda{I would remove the proof, if we keep the tim in this section, but if Ping wants it absolutely I am ready
%to keep it ...}}

We assume the existence and smoothness of $y(x)$ is proved.

Denote $b_0 := 2 \|(\del_y F(x_0,y_0))^{-1}\|$.
Since $\del_y F(x_0,y_0)$ is continuous in the neighborhood of $(x_0,y_0))$, we can assume that $\rho_x$ and $\rho_y$ are chosen small enough so that $\|(\del_y F(x,y))^{-1}\| \leq b_0$ in $B_x(x_0,\rho_x)\times B_y(y_0,\rho_y)$.

Denote $F_x := \delta_x F$, $F_y := \delta_y F$, $y_x := \delta_x y$.
We then have $F_x(x,y(x)) + F_y(x,y(x)) y_x(x) = 0$ for all $x\in B_x(x_0,\rho_x)$, or $y_x(x) = - F_y(x,y(x))^{-1} F_x(x,y(x))$.

To prove (b), estimate
\[
\|y_x(x)\|
\leq
\|F_y(x,y)^{-1}\| \|F_x(x,y)\| 
\leq b_0 \|F_x\|_\C.
\]

To show that $y_x(x)$ is Lipschitz, fix arbitrary $x_1, x_2 \in B_x(x_0,\rho_x)$, denote $y_1 = y(x_1)$ and $y_2 = y(x_2)$, and estimate
\begin{align*}
\|y_x(x_2) - y_x(x_1)\|
=~&
\|- F_y(x_2,y_2)^{-1} F_x(x_2,y_2)
+ F_y(x_1,y_1)^{-1} F_x(x_1,y_1)\|
\\
\leq~&
\|F_y(x_2,y_2)^{-1} [F_x(x_2,y_2) - F_x(x_1,y_1)]\|
\\~&+ \|[F_y(x_2,y_2)^{-1} - F_y(x_1,y_1)^{-1}] F_x(x_1,y_1)\|
\\
\leq~&
\|F_y(x_2,y_2)^{-1}\| \|F_x(x_2,y_2) - F_x(x_1,y_1)]\|
\\~&+ \|F_y(x_2,y_2)^{-1} [F_y(x_1,y_1) - F_y(x_2,y_2)] F_y(x_1,y_1)^{-1}\| \|F_x(x_1,y_1)\|
\\
\leq~&
b_0 |F_x|_{\C^{0,1}} (\|x_2-x_1\| + \|y_2 - y_1\|)
\\ ~&
+ b_0^2 |F_y|_{\C^{0,1}} (\|x_2-x_1\| + \|y_2 - y_1\|) \|F_x\|_\C.
\end{align*}
It remains to notice that $\|y_2 - y_1\| \leq b_0 \|F_x\|_\C \, \|x_2 - x_1\|$ due to part (b).
\end{proof}

\bibliographystyle{siam}
\bibliography{hqc.analysis}

\end{document}